\documentclass[ejs]{imsart}

\usepackage{amsthm} 
\usepackage[numbers]{natbib} 
\usepackage{myCommands}
\RequirePackage[colorlinks,citecolor=blue,urlcolor=blue]{hyperref}

\togglefalse{colcom}

\startlocaldefs
\newcommand{\ambdim}{D} 

\newtheorem{thm}{Theorem}[subsection]
\newtheorem{lem}[thm]{Lemma}
\newtheorem{prop}[thm]{Proposition}
\newtheorem{cor}[thm]{Corollary}

\theoremstyle{definition}
\newtheorem{defn}[thm]{Definition}
\newtheorem{conj}[thm]{Conjecture}

\theoremstyle{remark}
\newtheorem{rem}[thm]{Remark}

\numberwithin{equation}{section}

\endlocaldefs

\begin{document}

\begin{frontmatter}

\title{Linear Dimension Reduction\\ 
Approximately 
Preserving\\ 
a Function of the 1-Norm}
\runtitle{Preserving a Function of the 1-Norm} 

\begin{aug}
\author{\fnms{Michael P.} \snm{Casey}\corref{}\ead[label=e1]{mpcasey@alumni.duke.edu}}
\address{U. S. Air Force Research Laboratory,  mpcasey@alumni.duke.edu}
\affiliation{United States Air Force Research Laboratory}

\runauthor{Michael P. Casey}
\end{aug}

\begin{abstract}
For any finite point set in $D$-dimensional space equipped with the 1-norm, 
we present random linear embeddings to $k$-dimensional space, with a new metric, having the following properties. 
For any pair of points from the point set that are not too close, 
the 
distance between 
their images is a strictly concave increasing function of their original distance, up to multiplicative error. 
The target dimension $k$ need only be quadratic in the logarithm of the size of the point set to ensure the result holds with high probability. 
The linear embeddings are random matrices composed of standard Cauchy random variables, and the proofs rely on Chernoff bounds for sums of iid random variables. 
The new metric is translation invariant, but is not induced by a norm.
\end{abstract}

\begin{keyword}[class=MSC]
\kwd[Primary ]{60}
\kwd[; secondary ]{46B09, 46B85, 60E07, 60G50}
\end{keyword}

\begin{keyword}
\kwd{dimension reduction}
\kwd{embeddings of finite metric spaces}
\kwd{random projection}
\kwd{metric preserving function}
\kwd{Cauchy random variables}
\kwd{Cauchy projections}
\kwd{stable distributions}
\kwd{concentration of measure} 
\end{keyword}
%
\end{frontmatter}
%

\section{Introduction}

The Johnson-Lindenstrauss lemma~\citep{JohnsonLindenstrauss1984} states that
for a finite set of points $P\subset\R^\ambdim$ and $0<\epsilon<1$, there are random linear maps $F:\R^\ambdim\to\R^k$ satisfying, for any $x,y\in P$, 
\[
(1-\epsilon)\norm{x-y}_2\leq \norm{F(x)-F(y)}_2\leq (1+\epsilon)\norm{x-y}_2
\]
with high probability, provided $k=\Theta(\epsilon^{-2}\ln\abs{P})$. 
It is sufficient to draw the entries of $F$ $\iid$ sub-Gaussian~\citep{MatousekVariants2008}. 
These random linear projections have provided improved worst case performance bounds for many problems in theoretical computer science, machine learning, and numerical linear algebra.
Ailon and Chazelle~\citep{AilonChazelle2009} show how $F$ may be computed quickly and apply it to the approximate nearest-neighbor problem, working on the projected points $F(P)$. 
Vempala~\citep{VempalaRandom2004} gives a review of problems that may be reduced to analyzing a set of points $P\subset\R^\ambdim$, so that after the random projection $F:\R^\ambdim\to\R^k$ is applied, the recovery of approximate solutions is possible with time and space bounds depending on $k$, the target dimension, instead of $\ambdim$, the ambient dimension. 

In numerical linear algebra, 
Drineas et al.~\citep{DrineasStatLeverage2012} use the lemma to approximate the leverage scores of a given matrix $A$; such scores are used to inform subsampling schemes for $A$, resulting in sketches $\tilde{A}$ of smaller dimensions that preserve desired properties of $A$. 
Drineas and Mahoney~\citep{DrineasMahoneyRandNLA2016} give a further review of using randomness in numerical linear algebra. 

The Johnson-Lindenstrauss lemma is a metric embedding result; the map $F$ sends the finite metric space $P\subset\R^\ambdim$ induced by the 2-norm to a corresponding metric space $F(P)\subset\R^k$, also induced by the 2-norm, such that distances are preserved well. 
Ailon and Chazelle~\citep{AilonChazelle2009} also show that equipping the target space $\R^k$ with the 1-norm is also possible; the target dimension is still proportional to $\ln\abs{P}$, but the dependence on $\epsilon$ may be a bit worse. 
However, analogous results using the 1-norm on the domain do not hold.
For example, in~\citep{BrinkmanCharikarImpossibility2005} and~\citep{LeeNaorEmbedding2004}, specific $N$-point subsets of $\R^\ambdim$ equipped with the $1$-norm are shown to embed only in $\R^k$ with $k=N^{1/c^2}$ if one requires
\[
\norm{x-y}_1\leq \norm{F(x)-F(y)}_1\leq c\norm{x-y}_1. 
\]
In particular, Brinkman and Charikar~\citep{BrinkmanCharikarImpossibility2005} show the target dimension $k$ must be at least $N^{1/2-O(\epsilon\ln(1/\epsilon))}$ if one wants $c=1+\epsilon$. 

In light of these negative results, people have tried estimating $\norm{x-y}_1$ from the coordinates of $F(x)-F(y)$. 
When the entries of $F$ are $\iid$ standard Cauchy random variables, the coordinates are distributed $\iid$ like $\norm{x-y}_1 X$ with $X\drawn\Cauchy{1}$. 
The median of $\norm{x-y}_1\abs{X}$ is $\norm{x-y}_1$, so estimating the median 
from the coordinates of $F(x)-F(y)$ would estimate the distance this way.
Indyk~\citep{IndykStable2006} considers the sample median as an estimator, 
while Li, Hastie, and Church~\citep{LiHastieChurch2007} consider 1-homogeneous functions of these coordinates for estimators. 
None of the estimators considered are metrics on $\R^k$. 
For nearest neighbor methods, we should like to have a metric on the target space $\R^k$ and prefer a low number of coordinates for each point.  

Relaxing the problem as follows,
we wish to find linear maps $F:\R^\ambdim\to\R^k$ 
satisfying, for any $x,y\in P$, 
\[
(1-\epsilon)\mu(\norm{x-y}_1)
\leq \rho(F(x),F(y))
\leq (1+\epsilon)\mu(\norm{x-y}_1) 
\]
with high probability. 
We have changed the metric on $\R^k$ to $\rho$ instead of the one induced by the 1-norm, and we have introduced a nonlinear function $\mu$ in place of the identity function. 
We want $k=\Theta(\epsilon^{-2}\ln^c\abs{P})$, with $c<4$ or better.  

Here, $\mu:\R_+\to\R_+$ is a concave increasing function with $\mu(0)=0$. 
Such $\mu$ are called \qt{metric preserving} by Corazza~\citep{CorazzaIntroduction1999}, 
for the following reason:
\[
\mu(\norm{x-y}_1)\leq \mu(\norm{x-z}_1)+\mu(\norm{z-y}_1)
\qtext{for any} x,y,z\in\R^\ambdim,
\]
that is, they admit a new metric on the space that is \qt{compatible} with the old one.
In particular, spheres for the new metric about a particular point $y\in\R^\ambdim$, that is, the level sets
$\set{x\in\R^\ambdim\st \mu\after\norm{x-y}_1=t},$
look like scaled versions of spheres for the 1-norm (crosspolytopes) about that point; 
the scaling however is nonlinear. 
The 1-norm is used here as an example, but any other input metric will still satisfy the triangle inequality under such $\mu$.
Not all metric preserving functions are concave increasing, but such a choice ensures the new metric generates the same topology as the old one.

For us, the linear map $F:\R^\ambdim\to\R^k$ will have entries $F_{ij}\overset{\iid}{\drawn}\Cauchy{1}$, and we introduce the metric $\rho$ on $\R^k$ using an auxiliary function $\xi$:
\[
\rho(x,y):=\frac1{k}\sum_{i=1}^k \xi(\abs{x_i-y_i})
\]
with
\[
\xi(\lambda):=\ln(1+\sqrt{\lambda})+\frac1{2}\ln(1+\lambda)
\qtext{and}
\mu(\lambda):=\E\xi(\lambda F_{11})
\]
for $\lambda>0$. 
Our main theorem has several regimes depending on how big $\norm{x-y}_1$ can be. 
(See theorems~\ref{thm:BigRegime}, \ref{thm:MediumRegime}, and \ref{thm:SmallRegime}.) 
However, the primary result is as follows.
\begin{thm}\label{thm:basicNonlinearJL1Norm}
Let $F$, $\rho$, and $\mu$ be as above. 
Given $N$ points $P\subset\R^\ambdim$ and $\epsilon\in(0,1)$, 
\[
\mu\left(\frac{\norm{x-y}_1}{1+\epsilon}\right)
\leq \rho(F(x),F(y))
\leq \mu((1+\epsilon)\norm{x-y}_1) 
\]
for all $x,y\in P$ with $\norm{x-y}_1\geq \sqrt{1+\epsilon}$, 
provided
\[
k=\frac{C}{\epsilon^2(1-\epsilon)^2}\ln N. 
\]
\end{thm}

Independent of its interest as an analog of the Johnson-Lindenstrauss lemma, theorem~\ref{thm:basicNonlinearJL1Norm} also contributes to the study of $p$-stable projections. 
In fact, we make the following conjecture for $1<p<2$ upon replacing the entries $F_{ij}$ of $F$ by $\iid$ standard $p$-stable random variables and setting $\mu(\lambda)=\E\xi(\lambda F_{11})$. 
Just like the~\ref{thm:basicNonlinearJL1Norm}, the conjecture could have several parts based on how large $\norm{x-y}_p$ is, but the primary conjecture is as follows.
\begin{conj}
With $F$ and $\mu$ modified as above, and $\rho$, $\epsilon$, and $k$ as in theorem~\ref{thm:basicNonlinearJL1Norm}, the following bound holds
\[
\mu\left(\frac{\norm{x-y}_p}{1+\epsilon}\right)
\leq \rho(F(x),F(y))
\leq \mu\big((1+\epsilon)\norm{x-y}_p\big) 
\]
for all $x,y\in P$ with $\norm{x-y}_p=\Omega(1)$. 
\end{conj}
The setup for the proof would be the same as for theorem~\ref{thm:basicNonlinearJL1Norm}, relying on 1st and 2nd moment estimates for $\xi(\lambda\abs{W})$; however, because the density for a $p$-stable random variable $W$ is only implicitly defined, the needed 1st and 2nd moment estimates are not so straightforward, but could be empirically found on the computer using methods such as~\cite{ChambersMallowsStuckStable1976} to draw the $p$-stable random variables. 
This approach, in which we directly project the points from $\R^\ambdim$, may be contrasted to embedding $\ell_p^\ambdim\into \ell_1^n$ and applying theorem~\ref{thm:basicNonlinearJL1Norm} there. 
Pisier~\citep{PisierDimension1983} (see also \citep[chapter~8]{milman_asymptotic_1986} and \citep[chapter~9]{LedouxTalagrandProbability1991}) shows that such embeddings exist with distortion $(1+\epsilon)$, with $n$ proportional to $\ambdim$ and depending on $p$ and $\epsilon$.

\section{Overview of the Proof}\label{sec:Outline}
In this section, we explain the choices for the function $\xi$ and the metric $\rho$, as well as the use of Cauchy random variables, outlining the proof along the way. 

Consider a point $v\in \R^\ambdim$. 
The 1-stability of the Cauchy distribution dictates that the coordinates of the projected point $F(v)$ are Cauchy distributed: $F(v)_j\drawn \norm{v}_1 X_j$ with $X_j\overset{\iid}{\drawn}\Cauchy{1}$. 
The metric $\rho$ is then an empirical mean:
\[
\rho(F(v),0)=\frac1{k}\sum_{j=1}^k\xi(\norm{v}_1 X_j), 
\]
and if we marginalize out the Cauchy dependence, we recover the deterministic function $\mu$ of $\norm{v}_1$: 
\[
\E\rho(F(v),0)=\E \xi(\norm{v}_1 X)=:\mu(\norm{v}_1)
\qtext{for} 
X\drawn\Cauchy{1}. 
\]

We can now outline the proof as follows: let $x-y=v\in \R^\ambdim$. 
The projection map $F:\R^\ambdim\to\R^k$ is linear and the metric $\rho$ is translation invariant, so our goal is to show
$\mu(\norm{v}_1)\approx \rho(F(v),0)$
or upon setting $\norm{v}_1=\lambda$,  
\[
\mu(\lambda)\approx \frac1{k}\sum_{j=1}^k\xi(\lambda X_j)
\]
with high probability.
As usual, we use the exponential Markov inequality and the $\iid$ assumption to estimate
\begin{gather*}
\Prob\set{\frac1{k}\sum_{j=1}^k\xi(\lambda\abs{X_j})-\mu(\lambda)>t}\leq \left(e^{-st}\E e^{s\big(\xi(\lambda\abs{X})-\mu(\lambda)\big)}\right)^k 
\end{gather*}
with a similar setup for the lower tail. 
However, Cauchy random variables $X$ only have finite \emph{fractional} moments, 
\[
\E\abs{X}^b<\infty \qtext{only for} \abs{b}<1, 
\]
so the presence of $\xi(\lambda\abs{X})$ in the exponential requires $\xi(\lambda)=c\ln(o(\lambda))$ when $\lambda$ is large. 
Our choice of $\xi$ ensures this behavior:
\[
\xi(\lambda)=\ln(1+\sqrt{\lambda})+\frac1{2}\ln(1+\lambda)\leq 2\ln(1+\sqrt{\lambda}),
\]
while the presence of the \qt{1+} in the logarithms ensures $\xi$ is nonnegative, increasing, and sends 0 to 0. 
The function $\xi$ is thus subadditive and preserves the triangle inequality:
\[
\xi(\abs{x_i-y_i})\leq \xi(\abs{x_i-z_i})+\xi(\abs{z_i-y_i}),
\]
ensuring $\rho$ is a metric on $\R^k$. 
Because $\mu$ is the expectation of $\xi$, it inherits these properties, so that $\mu\after\norm{}_1$ induces a metric on the original space $\R^\ambdim$. 

We show in sections~\ref{sec:UpperTails} and~\ref{sec:LowerTails} that our tail bounds take the following form: To be concrete, here is the upper tail case, but the other lower tail cases are similar
\begin{equation}\label{eqn:tailBoundTemplate}\tag{$\diamond$}
\min_s e^{-s\Delta}\E e^{s\big(\xi(\lambda\abs{X})-\mu(\lambda)\big)}
\leq \exp\left(-\frac{\Delta^2}{4(V^2+A)}\right)
\end{equation}
with $\Delta$ depending on $\mu(\lambda)$, the function $V^2$ giving an upper bound for the 2nd moment or the variance of $\xi(\lambda \abs{X})$, and the auxiliary function $A(\lambda)$, derived from tail estimates for $\xi(\lambda \abs{X})$. 
The particular form of $\xi$ was chosen to give explicit control over all these quantities as $\lambda$ varies, allowing us to obtain bounds on equation~\eqref{eqn:tailBoundTemplate} that only weakly depend on $\lambda$. 

We arrive at the particular form~\eqref{eqn:tailBoundTemplate} for the tail bounds by estimating the moment generating function as follows, taking the upper tail as an example:
with $Y=\xi(\lambda\abs{X})-\mu(\lambda)$, we split $\E\exp(sY)$ into two terms and desire each to be bounded by something quadratic in $s$:
for the 1st term, using a 2nd order Taylor expansion for the exponential,
\[
\E\exp(sY)\I\set{sY\leq 1}\leq 1+s^2\E Y^2\leq 1+s^2V^2
\]
while for the 2nd term, we use integration by parts, eventually showing 
\[
\E\exp(sY)\I\set{sY> 1}\\
=e\Prob\set{Y>1/s}+\int_1^\infty e^t\Prob\set{Y>t/s}\,dt
\leq s^2 A(\lambda). 
\]
We can show the integrand decays exponentially in $t$ using our choice of $\xi$ and the explicit density for the Cauchy distribution:
\begin{align*}
\Prob\set{Y>t/s}&=\Prob\set{\xi(\lambda\abs{X})> \mu(\lambda)+t/s}\\
&\leq \Prob\set{2\ln(1+\sqrt{\lambda\abs{X}})>\mu(\lambda)+t/s}\leq C(\lambda)e^{-t/s}  
\end{align*}
with $C$ depending on $\lambda$ and $\mu(\lambda)$. 
We can then combine these estimates and optimize in $s$:
\begin{align*}
\min_s e^{-s\Delta}&\E\exp(sY)
\leq \min_s e^{-s\Delta}\exp(s^2V^2+s^2A(\lambda))
=\exp\left(\frac{-\Delta^2}{4(V^2+A(\lambda))}\right), 
\end{align*}
using $s=\Delta/(2(V^2+A(\lambda)))$. 

The tail probabilities now have the form
\[
\Prob\set{\abs{\frac1{k}\sum_{j=1}^k\xi(\lambda\abs{X_j})-\mu(\lambda)}>\Delta}
\leq 2\exp\left(-k\frac{\Delta^2}{4(V^2+A(\lambda))}\right)
\]
for a single $\lambda$ corresponding to a single vector  $v=x-y\in\R^\ambdim$. 
There are at most $\binom{N}{2}$ pairs of points from $P$, so we would want to choose the target dimension as
\[
k=(c+2)\ln(N)\frac{4(V^2+A(\lambda))}{\Delta^2}
\]
to ensure with probability at least $1-N^{-c}$, 
\[
\abs{\frac1{k}\sum_{j=1}^k\xi(\lambda\abs{X_j})-\mu(\lambda)}\leq\Delta 
\]
for all pairs of points simultaneously. 
However, the error $\Delta$ and the target dimension $k$ both depend on $\lambda$, so we require uniform estimates for these quantities. 
We find these by breaking up the possible values for $\lambda$ into three regimes: small, medium, and big
\begin{center}
\includegraphics[scale=.5]{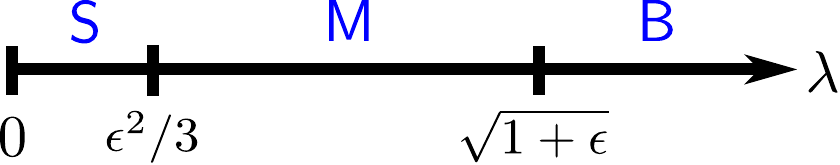}
\end{center} 
Our choice of $\xi$ provides an explicit function for $\mu:=\E\xi(\lambda\abs{X})$, (lemma~\ref{lem:XiFirstMoment})
\[
\mu(\lambda):=\frac1{2}\ln(1+\lambda^2)+\atanh\left(\frac{\sqrt{2\lambda}}{1+\lambda}\right) 
\qtext{with}
\atanh(x):=\sum_{j=0}^\infty \frac{x^{2j+1}}{2j+1} 
\]
for $\abs{x}<1$. 
The big regime has $\mu$ behaving like the log term, while the medium and small regimes have it behaving like $\Theta(\sqrt{2\lambda})$. 
The choice of $\xi$ also gives us a bound on the variance (corollary~\ref{rem:VarianceBounded})
\[
\Var(\xi(\lambda\abs{X}))
\leq \min\set{\frac{\pi^2}{2},2\E\ln(1+\lambda\abs{X})}. 
\]
The constant bound, independent of $\lambda$, is used for the big regime, 
while the expectation bound provides finer control on the variance when $\lambda$ is small, via another explicit function, lemma~\ref{lem:ExplicitEpxpectationLogLinearSimplerFunctions}, of $\lambda$. 

For the big regime, taking $\Delta$ as 
\[
\mu((1+\epsilon)\lambda)-\mu(\lambda)
\qtext{and}
\mu(\lambda)-\mu((1+\epsilon)^{-1}\lambda),
\]
both of which are bounded by lemma~\ref{lem:UpperDeviationEstimate}, 
together with the constant bound for the variance give theorem~\ref{thm:BigRegime}, as $A(\lambda)$ is bounded here. 

For the medium and small regimes, we take $\Delta=\epsilon\mu(\lambda)$ and use  
corollary~\ref{cor:2ndMomentRatioBound} to bound $V^2/\mu(\lambda)^2$.
The split between medium and small regimes occurs because of the $\ln(\lambda)$ term in that ratio: the target dimension has $\Delta^2\sim\lambda$ on the bottom, while the upper tail bound~\ref{lem:GenUpp} required $s^\ast$ to only have $\Delta$ on top
\[
s^\ast=\Delta/(2(V^2+A(\lambda)))<1/2. 
\]
This mismatch in powers of $\Delta$ forces us to choose a cutoff $\lambda$; 
because $A(\lambda)$ (and $V^2$) have terms proportional to $\lambda$, the above inequality can only hold for $\lambda$ not too small. 
This gives the $\ln(\epsilon^2/3)$ term in theorem~\ref{thm:MediumRegime} for the medium regime. 

For the small regime, there is no such restriction on $s$ for the lower tail bound~\ref{lem:LowerTailSmall}, but the target dimension still grows like $\ln(1/\lambda)$ as $\lambda$ decreases (See lemma~\ref{lem:LowerDeviation}.).  
We stop that growth by fixing a particular $\lambda_0$, showing that for all smaller $\lambda$, the $(1-\epsilon)$ error has a suitable replacement in theorem~\ref{thm:SmallRegime}. 
The key is lemma~\ref{lem:SuppositionBoundOnMaxAndUseOfHomog}: we choose $\lambda_0$ so that $\lambda_0\max_i\abs{X_i}<1/6$ with high probability, making \emph{both} $\xi(\lambda\abs{X_i})$ and $\mu(\lambda)$ behave like $\Theta(\sqrt{\lambda})=\sqrt{\eta}\Theta(\sqrt{\lambda_0})$ for $\lambda=\eta\lambda_0$ with $\eta\in(0,1)$. 
Because $\lambda_0$ turns out to be $\Theta(1/(N^{c+2}k))$, the $-\ln(\lambda_0)$ in the target dimension forces $k=\Theta(\epsilon^{-2}\ln^2(N^{c+2}))$, a quadratic dependence on $\ln(N)$. 

We finish the proofs in the next section, while the upper and lower tail estimates are provided in sections~\ref{sec:UpperTails} and~\ref{sec:LowerTails}. 
We collect the estimates on the 1st and 2nd moments in appendix~\ref{sec:Moments}, and ancillary identities for those estimates in appendix~\ref{app}.

\section{Finishing the Proof}
We now tie down the target dimension $k$. Recall $P$ is a set of $N$ points in $\R^D$, and $F:\R^D\to\R^k$ is a matrix of $\iid$ $\Cauchy{1}$ entries. 
In what follows, the estimates are not sharp. 
\begin{thm}[Big Regime]\label{thm:BigRegime}
For $\epsilon\in(0,1)$ and $\norm{x-y}_1\geq \sqrt{1+\epsilon}$, 
\[
\mu\left(\frac{\norm{x-y}_1}{1+\epsilon}\right)
\leq \rho\big(F(x),F(y)\big)
\leq \mu\big((1+\epsilon)\norm{x-y}_1\big)
\]
for all $x,y\in P$ with probability at least $1-N^{-c}$ provided
\[
k\geq \frac{C}{\epsilon^2(1-\epsilon)^2}\ln(N^{c+2})
\qtext{with}
C=64\left(\frac{\pi^2}{2}+\frac{16\sqrt{2}}{e\pi}e^{\atanh{1/\sqrt{2}}}\right). 
\]
\end{thm}
\begin{rem}
The constants are not expected to be sharp; $C$ is computed so that $k$ is uniformly bounded with respect to $\norm{x-y}_1\geq\sqrt{1+\epsilon}$. 
\end{rem}
\begin{proof}
Let $\lambda=\norm{x-y}_1$. 
We want to use the lower and upper tail estimates from lemmas~\ref{lem:LowerTailBig} and~\ref{lem:GenUpp}, so it remains to verify
\[
s^\ast=\frac{\Delta}{2(V^2+A(\lambda))}\leq 1/2 
\]
with $\Delta$ either
\[
\mu(\lambda)-\mu\big((1+\epsilon)^{-1}\lambda\big)
\qtext{or}
\mu\big((1+\epsilon)\lambda\big)-\mu(\lambda). 
\]
By lemma~\ref{lem:UpperDeviationEstimate}, the differences $\Delta$ are at most $\epsilon$ for $\lambda\geq\sqrt{1+\epsilon}$, 
while the upper bound for the variance of $\xi(\lambda\abs{X})$ is $V^2=\pi^2/2$ by corollary~\ref{rem:VarianceBounded}.
Because $\epsilon<1$, we then certainly have $s^\ast<1/2$. 

As explained in section~\ref{sec:Outline}, the target dimension $k$ is chosen to ensure the union bound is at most $N^{-c}$ for both tails combined.
The choice of $C$ comes from the lower bound for the $\Delta$'s from lemma~\ref{lem:UpperDeviationEstimate} and the larger of the two $A(\lambda)$ functions in lemmas~\ref{lem:LowerTailBig} and~\ref{lem:GenUpp}.
\end{proof}

\begin{thm}[Medium Regime]\label{thm:MediumRegime}
For $\norm{x-y}_1\in[\epsilon^2/3,\sqrt{1+\epsilon}]$ and 
$\epsilon\in(0,1)$, 
\[
(1-\epsilon)\mu\big(\norm{x-y}_1\big)
\leq \rho\big(F(x),F(y)\big)
\leq (1+\epsilon)\mu\big(\norm{x-y}_1\big)
\]
for all $x,y\in P$ with probability at least $1-N^{-c}$ provided
\[
k=\frac{4\ln(N^{c+2})}{\epsilon^2}\left(C+\frac{8}{\pi}\big(1-\ln(\epsilon^2/3)\big)\right)
\qtext{with $C$ bounded.} 
\]
\end{thm}
\begin{rem}
We have not been able to establish an upper bound result 
\[
\rho(F(x),F(y)) \leq (1+\epsilon)\mu(\norm{x-y}_1)
\]
with high probability when $\norm{x-y}<\epsilon^2/3$. 
Our proofs break down or require a much higher estimate for the target dimension $k$. 
We conjecture that $k=O(\ln^2(N^c)/\epsilon^2)$ still suffices, in light of theorem~\ref{thm:SmallRegime} for the small regime. 
\end{rem}
\begin{proof}
With $\lambda=\norm{x-y}_1$, we take $\Delta=\epsilon\mu(\lambda)$. 
By lemma~\ref{lem:LowerDeviation}, the lower bound for $\rho(F(x),F(y))$ requires an initial estimate for the target dimension of
$\tilde{k}=2\ln(N^{c+2})\epsilon^{-2}C(1-\ln(\lambda^\ast))$
with $\lambda^\ast$ the smallest $\lambda$ we wish to consider. 
The upper bound will force our choice of $\lambda^\ast$. 

We now want to use the upper tail estimate from lemma~\ref{lem:GenUpp}. 
It remains to check
\[
s^\ast=\frac{\epsilon\mu(\lambda)}{2(V^2+A(\lambda))}\leq 1/2, 
\]
and it suffices to show $\epsilon\mu(\lambda)/A(\lambda)\leq 1$. 
With $b=16e/(\pi(e-1)^2)\approx 4.7$ from $A(\lambda)$, 
we use lemma~\ref{lem:BoundsOnAtanhTermXiFirstMoment} for the upper bound for $\mu(\lambda)$ to find, after some estimation,
\[
\frac{\epsilon\mu(\lambda)}{A(\lambda)}
\leq \epsilon\frac{\sqrt{1+\lambda^2}}{b}\left(\frac{3}{\sqrt{2}}\frac1{(1+\lambda)\sqrt{\lambda}}+\frac{\lambda}{2}\right)
\leq \begin{cases}
.7\epsilon &\text{if $\lambda\in[1,\sqrt{1+\epsilon}]$}\\
1 &\text{if $\lambda\in[\epsilon^2/3,\,1]$} 
\end{cases},
\]
recalling $\epsilon\in(0,1)$. 

Using the expression for $A(\lambda)$, we now have the following estimate for the target dimension. 
Because $V^2$ is an estimate on the variance now, we can remove the $1+$'s from corollary~\ref{cor:2ndMomentRatioBound} to find
\[
k=\frac{4\ln(N^{c+2})}{\epsilon^2}\left(\frac{9}{2}(\sqrt{1+\epsilon}+4/\pi)+
    2+\frac{8}{\pi}\big(1-\ln(\epsilon^2/3)\big)
+\frac{b}{2}(1+\sqrt{1+\epsilon})^2\right)
\]
using $\mu(\lambda)\geq \sqrt{2\lambda}/(1+\lambda)$ from remark~\ref{rem:MuApproxSmallScales}. 
The $b$ dependent term here is enough to ensure $k\geq\tilde{k}$, so both sides of the inequality for $\rho(F(x),F(y))$ hold with high probability and this dimension $k$. 
\end{proof}

The following two lemmas lead to theorem~\ref{thm:SmallRegime}, which shows that a lower bound for $\rho(F(x),F(x))$ continues to hold for \emph{all} $\norm{x-y}_1<2$. 
\begin{lem}\label{lem:LowerDeviation}
For $\epsilon\in(0,1)$ and $0<\norm{x-y}_1\in[\lambda^\ast,2)$, 
\[
(1-\epsilon)\mu\big(\norm{x-y}_1\big)
\leq \rho\big(F(x),F(y)\big)
\]
for all $x,y\in P$ with probability at least $1-N^{-c}$ provided
\[
k=\frac{2\ln(N^{c+2})}{\epsilon^2}
    \left(C+\frac{8}{\pi}\big(1-\ln(\lambda^\ast)\big)\right)
\qtext{with $C$ bounded.}
\]
\end{lem}
\begin{rem}
The estimates are not sharp. 
\end{rem}
\begin{proof}
With $\lambda=\norm{x-y}_1$, we take $\Delta=\epsilon\mu(\lambda)$. 
Using corollary~\ref{cor:2ndMomentRatioBound} and the lower tail esimate from lemma~\ref{lem:LowerTailSmall}, the target dimension is
\[
k=\frac{2\ln(N^{c+2})}{\epsilon^2}\max\set{
    1+(9/2)(2+4/\pi),\,
    3+(8/\pi)\big(1-\ln(\lambda^\ast\big)}
\]
to ensure the bound holds with probability at least $1-N^{-c}$ for all pairs of points. 
\end{proof}

\begin{lem}\label{lem:SuppositionBoundOnMaxAndUseOfHomog}
For $1\leq i\leq k$, let $X_i\overset{\iid}{\drawn}\Cauchy{1}$. 
For $0<\epsilon<1$ and $0<\lambda_0\leq 1$, suppose
\[
(1-\epsilon)
\mu(\lambda_0)
\leq \frac1{k}\sum_{i=1}^k \xi(\lambda_0\abs{X_i})
\]
and $\lambda_0\max_i\abs{X_i}\leq c_0\leq 1/6$. 

Then if $0<\eta<1$, the same $X_i$ also satisfy
\[
(1-\epsilon')
\mu(\eta\lambda_0)
\leq \frac1{k}\sum_{i=1}^k \xi(\eta\lambda_0\abs{X_i})
\]
with $\epsilon'$ depending on $\epsilon$, $c_0$, and $\lambda_0$. 
If $\lambda_0\leq \epsilon^2$, then we can have
\[
1-\epsilon'=(1-\epsilon)\frac{1-\epsilon^2/\sqrt{2}}{1+3\epsilon^2}. 
\]
\end{lem}
\begin{rem}
Analogous upper bounds are also possible, with a similar proof. 
\end{rem}
\begin{proof}
A fourth order Taylor expansion with Lagrange remainder shows 
\[
\sqrt{a}\leq \xi(a)\leq \sqrt{a}(1+a/2)
\qtext{for} 0\leq a\leq 1/6. 
\]
Because $\max_i{\abs{X_i}}\leq c_0\leq 1/6$ and $0<\eta<1$, we invoke the above inequality twice to find
\[
\frac{\sqrt{\eta}}{1+\lambda_0c_0/2}\xi(\lambda_0\abs{X_i})
\leq \frac{\sqrt{\eta}}{1+\lambda_0\abs{X_i}/2}\xi(\lambda_0\abs{X_i})
\leq \xi(\eta\lambda_0\abs{X_i}). 
\]
By assumption, summing over $i$ and dividing by $k$ yields
\[
(1-\epsilon)\frac{\sqrt{\eta}}{1+\lambda_0c_0/2}\mu(\lambda_0)
\leq \frac1{k}\sum_{i=1}^k \xi(\eta\lambda_0\abs{X_i})
\]
We finish by using remark~\ref{rem:MuApproxSmallScales} (twice) to \qt{absorb} $\sqrt{\eta}$ into $\mu$,
\[
\sqrt{\eta}\mu(\lambda_0)\geq \frac{\sqrt{2\eta\lambda_0}}{1+\lambda_0}
\geq \frac{\mu(\eta\lambda_0)}{1+\lambda_0}\frac{1+\eta\lambda_0}{1+\eta\lambda_0(1+1/\sqrt{2})} 
\geq \frac{\mu(\eta\lambda_0)}{1+\lambda_0}(1-\eta\lambda_0/\sqrt{2})
\]
\end{proof}

\begin{thm}[Small Regime]\label{thm:SmallRegime}
For $\epsilon\in[N^{-(c+2)/2},1]$ and all $\norm{x-y}_1<2$, 
the following bound holds: 
\[
(1-\epsilon)\frac{1-\epsilon^2/\sqrt{2}}{1+3\epsilon^2}\mu(\norm{x-y}_1)
\leq \rho(F(x),F(y)),
\]
with probability at least $1-N^{-c}$, provided
\[
k=\frac{C\ln^2(N^{c+2})}{\epsilon^2}. 
\]
\end{thm}
\begin{proof}
We can use lemma~\ref{lem:LowerDeviation} with $\lambda^\ast=\lambda_0$ to cover all distances $\norm{x-y}_1$ down to $\lambda_0$. 
We then choose $\lambda_0$ in order to extend the lower bound to distances smaller than $\lambda_0$, using lemma~\ref{lem:SuppositionBoundOnMaxAndUseOfHomog}. 

Concretely, recall from section~\ref{sec:Outline} that because $F$ is a linear map of $\iid$ Cauchy entries, 
\[
\rho(F(x/\eta),F(y/\eta))\drawn \frac1{k}\sum_{i=1}^k\xi(\lambda_0\abs{X_i})
\qtext{and}
\rho(F(x),F(y))\drawn \frac1{k}\sum_{i=1}^k\xi(\eta\lambda_0\abs{X_i})
\] 
with \emph{the same} $X_i\overset{\iid}{\drawn}\Cauchy{1}$. 
Let $Z=\max_{1\leq i\leq k}\abs{X_i}$.  
To use lemma~\ref{lem:SuppositionBoundOnMaxAndUseOfHomog}, we just need 
to ensure $\lambda_0Z\leq 1/6$ with high probability. 
By the independence of the $X_i$, 
\[
\Prob\set{\lambda_0Z\leq 1/6}
=\big(1-(2/\pi)\arctan(6\lambda_0)\big)^k
\geq 1-k(12/\pi)\lambda_0. 
\]
So set $\lambda_0=\pi/(12 N^{c+2}k)$.
Choosing $k$ according to lemma~\ref{lem:LowerDeviation} with $\lambda_0=\lambda^\ast$, we have the following inequality for the target dimension
\[
k\geq \frac{2\ln(N^{c+2})}{\epsilon^2}
    \left(C+\frac{8}{\pi}\big(1-\ln(\pi/(12 N^{c+2}k))\big)\right). 
\]
Taking $k=C\ln^2(N^{c+2})/\epsilon^2$ satisfies the above, provided $\epsilon^2>N^{-c-2}$, say. 
We now have the conditions of 
lemma~\ref{lem:SuppositionBoundOnMaxAndUseOfHomog} satisfied for all $\binom{N}{2}$ pairs of points, with probability at least $1-N^{-c}$, and $\lambda_0<\epsilon^2$.  
\end{proof}

\section{Upper Tails}\label{sec:UpperTails}
In the following lemmas, the estimates are not sharp. 
\begin{lem}[General Upper Tail]\label{lem:GenUpp}
With $Y=\E\xi(\lambda\abs{X})-\mu(\lambda)$ and $V^2\geq \Var(\xi(\lambda\abs{X}))$, 
\[
\min_s e^{-s\Delta}\E\exp(sY)\leq \exp\left(\frac{-\Delta^2}{4(V^2+A(\lambda))}\right) 
\]
and is minimized at $s^\ast$
with 
\[
A(\lambda)=\frac{16 e}{\pi (e-1)^2}\sqrt{\frac{\lambda^2}{1+\lambda^2}} \qtext{provided}
s^\ast=\frac{\Delta}{2(V^2+A(\lambda))}\leq 1/2. 
\]
\end{lem}
\begin{proof}
From the discussion in section~\ref{sec:Outline}, we just need to establish the $A(\lambda)$ function for the integration by parts terms. 
To ensure $\E\exp(sY)$ is finite, we require $s<1$. 
For $1>s>0$ and $t>1$, we then estimate, with $w=\mu+t/s$,
\begin{gather*}
\Prob\set{\xi(\lambda\abs{X})>w}
\leq \Prob\set{2\ln(1+\sqrt{\lambda\abs{X}})>w}
\leq \frac{2}{\pi}\arctan\left(\frac{\lambda}{(\exp(w/2)-1)^2}\right),
\intertext{and for all $w\geq w_0$ may be bounded by} 
\leq \frac{2}{\pi}\frac{\lambda}{(\exp(w/2)-1)^2}
\leq \frac{2\lambda}{\pi}\frac{e^{-w}}{(1-\exp(-w_0/2))^2} 
=:C(\lambda)e^{-t/s}. 
\end{gather*}
We shall choose $w_0$ and hence $C(\lambda)$ a bit later; note that $C(\lambda)$ contains the factor $e^{-\mu}$. 

With $Y=\xi(\lambda\abs{X})-\mu(\lambda)$, we can now estimate the integration by parts terms
\[
\E\exp(sY)\I\set{sY>1}=e\Prob\set{Y>1/s}+\int_1^\infty e^t\Prob\set{Y>t/s}\,dt
\]
as at most
\begin{gather*}
eC(\lambda)e^{-1/s}+C(\lambda)\int_1^\infty e^{t(1-1/s)}\,dt
=eC(\lambda)\frac{e^{-1/s}}{1-s}
\leq \frac{4}{e}C(\lambda)\frac{s^2}{1-s}
\end{gather*}
using $e^{-1/s}\leq (2/e)^2s^2$ and $s\in(0,1)$. 
Assuming $s\leq 1/2$, we can now write, for a suitable upper bound $A(\lambda)$, 
\[
\E\exp(sY)\leq 1+V^2s^2+\frac{8}{e}C(\lambda)s^2
\leq 1+s^2(V^2+A(\lambda)), 
\]
and we may optimize in $s$
\[
\min_s e^{-s\Delta}\E\exp(sY)\leq \min_s e^{-s\Delta+s^2(V^2+A)}
=\exp\left(\frac{-\Delta^2}{4(V^2+A)}\right)
\]
at 
\[
s^\ast=\frac{\Delta}{2(V^2+A)}.  
\]

It remains to choose $w_0$ and hence $A(\lambda)$. 
Recalling the formula for $\mu$ either from section~\ref{sec:Outline} or directly from lemma~\ref{lem:XiFirstMoment}, we can lower bound $w=\mu(\lambda)+t/s\geq (1/2)\ln(1+\lambda^2)+ 2\geq 2$ provided $s\leq 1/2$. 
Choosing $w_0=2$, 
\[
\frac{8}{e}C(\lambda)
\leq \frac{16}{\pi e}\frac{\lambda}{(1+\lambda^2)^{1/2}}\frac1{(1-1/e)^2}
=\frac{16 e}{\pi (e-1)^2}\sqrt{\frac{\lambda^2}{1+\lambda^2}}
=:A(\lambda). 
\]
\end{proof}

\section{Lower Tails}\label{sec:LowerTails}
Unlike for the upper tails, we can control the lower tails for the full range of $\lambda$. 
We address $\lambda$ bounded away from 0 using the same techniques as for the upper tail. 
The lower tail proof for small $\lambda$ simplifies because $\xi(\lambda\abs{X})$ is nonnegative, so that there is no restriction on optimizing $s$ in the moment generating function. 

In the following lemmas, the estimates are not sharp. 
\begin{lem}[Lower Tail, Big Regime]\label{lem:LowerTailBig}
With $Y=\E\xi(\lambda\abs{X})-\mu(\lambda)$ and $V^2\geq \Var(\xi(\lambda\abs{X}))$, 
\[
\min_s e^{-s\Delta}\E\exp(-sY)\leq \exp\left(\frac{-\Delta^2}{4(V^2+A(\lambda))}\right) 
\]
and is minimized at $s^\ast$ with 
\[
A(\lambda)
=\frac{16}{e\pi}e^{\atanh(1/\sqrt{2})}\sqrt{\frac{1+\lambda^2}{\lambda^2}}
\qtext{provided}
s^\ast=\frac{\Delta}{2(V^2+A(\lambda))}\leq 1/2.
\]
\end{lem}
\begin{proof}
Just as in the upper tail computations, 
\[
\E\exp(-sY)\I\set{-sY\leq 1}\leq \E(1-sY+(-sY)^2)=1+ V^2s^2. 
\]
and
\[
\E\exp(-sY)\I\set{-sY>1}
\leq e\Prob\set{-Y>1/s}+\int_1^\infty \exp(t)\Prob\set{-Y>t/s}\,dt. 
\]
We shall again determine the function $A(\lambda)$ by estimating a tail, but now it is the lower tail
\[
\Prob\set{-Y>t/s}=\Prob\set{\xi(\lambda\abs{X})<\mu-t/s}.  
\]
By subadditivity of $\sqrt{a}$, 
\[
\xi(a)=\ln(1+\sqrt{a})+\frac1{2}\ln(1+a)
\geq \ln(\sqrt{1+a})+\frac1{2}\ln(1+a)
=\ln(1+a). 
\]
We now can estimate
\[
\Prob\set{-Y>t/s}
<\frac{2}{\pi}\arctan\left(e^{-t/s}\frac{e^\mu}{\lambda}\right)
<\frac{2}{\pi}\frac{e^\mu}{\lambda}e^{-t/s}
=:C(\lambda)e^{-t/s}. 
\]
We can then upper bound the integration by parts terms just like in the proof for the upper tail lemma~\ref{lem:GenUpp}. 
Assuming $s\leq 1/2$, we choose an upper bound $A(\lambda)$ for $(8/e)C(\lambda)$ and arrive at 
\[
\min_s e^{-s\Delta}\E\exp(-sY)\leq \exp\left(\frac{-\Delta^2}{4(V^2+A(\lambda))}\right) 
\]
To find $A(\lambda)$, note that $\mu(\lambda)\leq \atanh(1/\sqrt{2})+(1/2)\ln(1+\lambda^2)$, so that
\[
\frac{8}{e}C(\lambda)
\leq \frac{16}{e\pi}e^{\atanh(1/\sqrt{2})}\sqrt{\frac{1+\lambda^2}{\lambda^2}}
=:A(\lambda) 
\]
which is bounded for $\lambda$ away from 0.
\end{proof}

\begin{lem}[Lower Tail, Small Regimes]\label{lem:LowerTailSmall}
With $Y=\E\xi(\lambda\abs{X})-\mu(\lambda)$ and $V^2\geq \E\xi(\lambda\abs{X})^2$, 
\[
\min_s e^{-s\Delta}\E\exp(-sY)\leq \exp\left(\frac{-\Delta^2}{2V^2}\right) 
\]
\end{lem}
\begin{proof}
Because $\xi(\lambda\abs{X})$ is nonnegative, we can use the 2nd order Taylor expansion of $\exp(-x)$ to write
\[
e^{-s\Delta+s\mu}\E\exp(-s\xi(\lambda\abs{X}))
\leq e^{-s\Delta+s\mu}(1-s\mu+\frac{s^2}{2}V^2)
=e^{-s\Delta+s^2V^2/2} 
\]
and we can then optimize $s$ in the usual way. 
\end{proof}

\section*{Acknowledgments}
This work was supported in part by Duke University while completing my Ph.D. thesis.
I should like to thank my advisor Professor Sayan Mukherjee for encouraging me in completing this work. 
%
I should also like to thank the anonymous referee, whose comments helped greatly streamline the paper. 
%
I should like to thank Mom, Dad, Katie, and everyone who has been praying for me throughout my time at Duke. 
%
I should finally like to thank the Blessed Virgin Mary, Saint Joseph, and the Holy Trinity for helping me be patient throughout this work.

\appendix
\section{The First and Second Moments}\label{sec:Moments}
Here we derive the explicit formula for $\mu(\lambda)=\E\xi(\lambda\abs{X})$ in lemma~\ref{lem:XiFirstMoment} and the upper bounds $V^2$ for $\Var(\xi(\lambda\abs{X}))$ in lemma~\ref{rem:VarianceBounded}. 
Some of this work is a bit tedious, but it will allow us to give explicit upper bounds on the target dimension
\[
k=C\frac{V^2+A}{\Delta^2}. 
\]
We need bounds for $\mu(\lambda)$ when $\lambda$ is small (lemma~\ref{rem:MuApproxSmallScales}) as well as lower bounds for the $\Delta$'s 
\[
\mu((1+\epsilon)\lambda)-\mu(\lambda)
\qtext{and}
\mu(\lambda)-\mu((1+\epsilon)^{-1}\lambda)
\]
when $\lambda$ is \qt{large} (lemma~\ref{lem:UpperDeviationEstimate}). 

Because the Cauchy density has particularly simple behavior when extended to the complex plane, we heavily rely on complex analysis techniques. 
We chose $\xi$ to be the linear combination 
\[
\xi(\lambda\abs{X})=\ln(1+\sqrt{\lambda\abs{X}})+\frac1{2}\ln(1+\lambda\abs{X})
\]
as it will simplify the estimates as well as be easy to compute using a pair of contour integrals. 
For both moments, the contour integral setup below will greatly facilitate computations; in particular, it will allow us to avoid estimating
\[
\E\ln^2(1+\sqrt{\lambda\abs{X}})\qtext{and}\E\ln^2(1+\lambda\abs{X})
\]
individually, which while possible, is not necessary for our results.
\begin{prop}[Contour Integral Setup]\label{lem:ContourIntegralSetup}
For $\lambda>0$, $b>0$, and $X\drawn\Cauchy{1}$, 
\begin{align*}
\E\ln^b(1+\sqrt{\lambda\abs{X}})
&=\ln^b(1+\sqrt{i\lambda})+\ln^b(1+\sqrt{-i\lambda})\\
    &\qquad-\frac1{2}\E\ln^b(1+i\sqrt{\lambda\abs{X}})
    -\frac1{2}\E\ln^b(1-i\sqrt{\lambda\abs{X}}). 
\end{align*}
\end{prop}
\begin{rem}
The task is then to simplify the complex logarithms on the right hand side when particular values of $b$ are chosen. We shall choose $b=1$ and $b=2$ in the next sections.
\end{rem}
\begin{proof}
We want to compute
\[
I(\lambda):=\E\ln^b(1+\sqrt{\lambda\abs{X}})
=\frac{2}{\pi}\int_0^\infty \frac{\ln^b(1+\sqrt{\lambda x})}{1+x^2}\,dx 
\]
via contour integration. 
Extending to $z\in\C-(-\infty,0]$, let
\[
f(z):=\frac{2}{\pi}\frac{\ln^b(1+\sqrt{\lambda z})}{1+z^2}
\]
which has simple poles at $z=\pm i$. 

We shall compute $I(\lambda)$ be using two different contours that both traverse the interval $[0,R]$ in the positive direction. 
Specifically, $\mathcal{C}^+$ is oriented counterclockwise, while
$\mathcal{C}^-$ is oriented clockwise, setting
\[
\mathcal{C}^+:=[0,R]\cup \mathcal{C}^+(R)\cup \mathcal{C}^+_\epsilon(R)
\qtext{and}
\mathcal{C}^-:=[0,R]\cup \mathcal{C}^-(R)\cup \mathcal{C}^-_\epsilon(R)
\]
with \qt{large} arcs 
\[
\mathcal{C}^\pm(R):=\set{Re^{\pm i\theta}\st 0\leq\theta\leq \pi-\epsilon}
\]
and segments rotating as $\epsilon\to 0$ to the negative real axis 
\[
\mathcal{C}^\pm_\epsilon(R):=\set{re^{\pm i(\pi-\epsilon)}\st R\geq r\geq 0}.
\]

Check that 
\[
\lim_{R\to\infty}\int_{\mathcal{C}^\pm(R)}f(z)\,dz=0. 
\]
Keeping in mind the orientations of the contours, the residue theorem dictates for $R>1$,
\begin{align*}
\int_{\mathcal{C}^+}f(z)\,dz &=2\pi i\res_{z=i}f(z)\\
&=2\pi i\lim_{z\to i}(z-i)\frac{2}{\pi}
    \frac{\ln^b(1+\sqrt{\lambda z})}{(z-i)(z+i)}
=2\ln^b(1+\sqrt{\lambda i}) 
\end{align*}
and similarly
\begin{align*}
\int_{\mathcal{C}^-}f(z)\,dz
&=-2\pi i\lim_{z\to -i}(z-(-i))\frac{2}{\pi}
    \frac{\ln^b(1+\sqrt{\lambda z})}{(z-i)(z+i)}
=2\ln^b(1+\sqrt{-i\lambda}). 
\end{align*}

It remains to show that 
\begin{gather*}
\lim_{R\to\infty}\lim_{\epsilon\to 0}
    \left(
    \int_{\mathcal{C}^+_\epsilon(R)}f(z)\,dz
    +\int_{\mathcal{C}^-_\epsilon(R)}f(z)\,dz\right)\\
=\E\ln^b(1+i\sqrt{\lambda\abs{X}})+\E\ln^b(1-i\sqrt{\lambda\abs{X}}). 
\end{gather*}
For these $\mathcal{C}^\pm_\epsilon(R)$ integrals, note that
\[
\sqrt{re^{\pm i(\pi-\epsilon)}}=\sqrt{r}e^{\mp i\epsilon/2}e^{\pm i\pi/2}
=\pm i\sqrt{r}e^{\mp i\epsilon/2}, 
\]
which approaches $\pm i \sqrt{r}$ when $\epsilon\to 0$. 
Consequently, when $z=re^{i(\pi-\epsilon)}=-re^{-i\epsilon}$, 
we can use the dominated convergence theorem to conclude
\begin{align*}
\lim_{\epsilon\to 0}\int_{\mathcal{C}^+_\epsilon(R)}f(z)\,dz
&=\lim_{\epsilon\to 0}\int_R^0 f(-re^{-i\epsilon})\,(-e^{-i\epsilon})dr\\
&=\lim_{\epsilon\to 0}\int_0^R
    \frac{2}{\pi}\frac{e^{-i\epsilon}\ln^b(1+i\sqrt{\lambda r}e^{-i\epsilon/2})}
        {1+r^2e^{-2i\epsilon}}\,dr\\
&=\int_0^R\frac{2}{\pi}\frac{\ln^b(1+i\sqrt{\lambda r})}{1+r^2}\,dr,
\end{align*}
checking that the integrand is bounded by a summable one when $\epsilon<\pi/8$ say. 
Sending $R\to\infty$ recovers
\[
\lim_{R\to\infty}\lim_{\epsilon\to 0}\int_{\mathcal{C}^+_\epsilon(R)}f(z)\,dz
=\E\ln^b(1+i\sqrt{\lambda\abs{X}}). 
\]
Similar reasoning applies to the $\mathcal{C}^-_\epsilon(R)$ integral to yield
\[
\lim_{R\to\infty}\lim_{\epsilon\to 0}\int_{\mathcal{C}^-_\epsilon(R)}f(z)\,dz
=\E\ln^b(1-i\sqrt{\lambda\abs{X}})
\]
Putting everything together, we have
\begin{gather*}
2\ln^b(1+\sqrt{i\lambda})+2\ln^b(1+\sqrt{-i\lambda})\\
=2I(\lambda)+\E\ln^b(1+i\sqrt{\lambda\abs{X}})+\E\ln^b(1-i\sqrt{\lambda\abs{X}})
\intertext{that is}
\ln^b(1+\sqrt{i\lambda})+\ln^b(1+\sqrt{-i\lambda})\\
\qquad=I(\lambda)+\frac1{2}\E\ln^b(1+i\sqrt{\lambda\abs{X}})+\frac1{2}\E\ln^b(1-i\sqrt{\lambda\abs{X}}) 
\end{gather*}
as claimed.
\end{proof}

\subsection{1st Moment}
Recall from definition~\ref{defn:ATanh} that $\atanh(x)$ may be defined by the power series
\[
\atanh(x)=\sum_{j=0}^\infty \frac{x^{2j+1}}{2j+1}
\qtext{for}
\abs{x}<1. 
\]
\begin{lem}\label{lem:XiFirstMoment}
If $\lambda>0$ and $X\drawn\Cauchy{1}$, then
\[
\E\ln(1+\sqrt{\lambda\abs{X}})
=\atanh\left(\frac{\sqrt{2\lambda}}{1+\lambda}\right)+\frac1{2}\ln(1+\lambda^2)
    -\frac1{2}\E\ln(1+\lambda\abs{X})
\]
that is,
\[
\mu(\lambda):=\E\xi(\lambda\abs{X})
=\atanh\left(\frac{\sqrt{2\lambda}}{1+\lambda}\right)+\frac1{2}\ln(1+\lambda^2).
\]
\end{lem}
\begin{proof}
Starting from lemma~\ref{lem:ContourIntegralSetup} with $b=1$, 
\begin{gather*}
\E\ln(1+\sqrt{\lambda\abs{X}})\\
=\ln(1+\sqrt{i\lambda})+\ln(1+\sqrt{-i\lambda})
    -\frac1{2}\E\ln(1+i\sqrt{\lambda\abs{X}})
    -\frac1{2}\E\ln(1-i\sqrt{\lambda\abs{X}}). 
\end{gather*}
By lemma~\ref{lem:LogToAtanhLogSplit} and the atanh addition formula~\ref{lem:ATanhAdditionFormula}, 
\begin{gather*}
\ln(1+\sqrt{\lambda i})+\ln(1+\sqrt{-\lambda i})\\
=\atanh(\sqrt{\lambda i})+\atanh(\sqrt{-\lambda i})
    +\frac1{2}\ln(1-(\sqrt{\lambda i})^2)+\frac1{2}\ln(1-(\sqrt{-\lambda i})^2)\\
=\atanh\left(\frac{\sqrt{i\lambda}+\sqrt{-i\lambda}}
        {1+\sqrt{i\lambda}\sqrt{-i\lambda}}\right)
    +\frac1{2}\ln(1-i\lambda)+\frac1{2}\ln(1-(-i\lambda))\\
=\atanh\left(\frac{\sqrt{2\lambda}}{1+\lambda}\right)
    +\frac1{2}\ln(1-(i\lambda)^2)\\
=\atanh\left(\frac{\sqrt{2\lambda}}{1+\lambda}\right)+\frac1{2}\ln(1+\lambda^2).
\end{gather*}
By remark~\ref{rem:LogShiftPM}, 
\[
\ln(1+i\sqrt{\lambda\abs{X}})+\ln(1-i\sqrt{\lambda\abs{X}})
=\ln(1-(i\sqrt{\lambda\abs{X}})^2)=\ln(1+\lambda\abs{X}). 
\]
Consequently, 
\[
\E\ln(1+\sqrt{\lambda\abs{X}})+\frac1{2}\E\ln(1+\lambda\abs{X})
=\atanh\left(\frac{\sqrt{2\lambda}}{1+\lambda}\right)+\frac1{2}\ln(1+\lambda^2)
\]
as claimed.
\end{proof}

We use the following lemma to show that $\mu(\lambda)=\Theta(\sqrt{\lambda})$ as well when $\lambda$ is small. 
\begin{lem}\label{lem:BoundsOnAtanhTermXiFirstMoment}
For $\lambda>0$, 
\[
\frac{\sqrt{2\lambda}}{1+\lambda}
<\atanh\left(\frac{\sqrt{2\lambda}}{1+\lambda}\right)
< \frac{\sqrt{2\lambda}}{1+\lambda}
    \left(1+\frac1{2}\ln\Big(1+\frac{2\lambda}{1+\lambda^2}\Big)\right)
<\frac{3}{\sqrt{2}}\frac{\sqrt{\lambda}}{1+\lambda}
\]
and approaches 0 as $\lambda\to\infty$. 
Further, for any $\lambda\leq \lambda_0\leq 1$, 
\[
\frac{\sqrt{2\lambda}}{1+\lambda}
<\atanh\left(\frac{\sqrt{2\lambda}}{1+\lambda}\right)
<\frac{\sqrt{2\lambda}}{1+\lambda}
    \left(1+\frac1{2}\ln\Big(1+\frac{2\lambda_0}{1+\lambda_0^2}\Big)\right). 
\]
\end{lem}
\begin{rem}\label{rem:MuApproxSmallScales}
By lemma~\ref{lem:XiFirstMoment}, we now also have the bound
\[
\frac{\sqrt{2\lambda}}{1+\lambda}
\leq \mu(\lambda)
\leq \frac{\sqrt{2\lambda}}{1+\lambda}
    \left(1+\frac{\lambda_0}{1+\lambda_0^2}\right)+\frac{\lambda^2}{2}
\leq \frac{3}{\sqrt{2}}\frac{\sqrt{\lambda}}{1+\lambda}+\frac{\lambda^2}{2}. 
\]
using $\ln(1+x)\leq x$ twice. 
\end{rem}
\begin{proof}
The limit for large $\lambda$ is immediate. 
From the power series for $\atanh$, conclude $\atanh(x)>x$ for $x>0$. 
We can also give the upper bound
\begin{gather*}
\atanh(x)=\sum_{j=0}^\infty \frac{x^{2j+1}}{2j+1}
=x\sum_{j=0}^\infty \frac{(x^2)^j}{2j+1}
\leq x\left(1+\frac1{2}\sum_{j=1}^\infty \frac{(x^2)^j}{j}\right)\\
=x\left(1-\frac1{2}\ln(1-x^2)\right). 
\end{gather*}
So, 
\[
\atanh\left(\frac{\sqrt{2\lambda}}{1+\lambda}\right)
\leq \frac{\sqrt{2\lambda}}{1+\lambda}
    \left(1+\frac1{2}\ln\Big(1+\frac{2\lambda}{1+\lambda^2}\Big)\right)
\]
Noting that $\lambda/(1+\lambda^2)$ is strictly increasing for $\lambda\in(0,1)$, we can fix the $\lambda^2$ term at a particular constant.  
\end{proof}

\subsection{Estimating Deviations of the Mean}
We derive the estimates used in the large scale concencentration proofs given above. 
Both differences
\[
\mu((1+\epsilon)\lambda)-\mu(\lambda)
\qtext{and}
\mu(\lambda)-\mu((1+\epsilon)^{-1}\lambda)
\]
are controlled by lemma~\ref{lem:UpperDeviationEstimate} by requiring $\lambda\geq \sqrt{1+\epsilon}$. 
Because 
\[
\mu(\lambda)=\atanh\left(\frac{\sqrt{2\lambda}}{1+\lambda}\right)+\frac1{2}\ln(1+\lambda^2)
\]
both deviations will be sums of two terms, an $\atanh$ term and a $\ln$ term.
\begin{lem}\label{lem:UpperDeviationEstimate}
For $1\leq a$ 
and $1/\sqrt{a}\leq \lambda$,
\[
a-1
>\mu(a\lambda)-\mu(\lambda)
\geq \frac{a-1}{4}(1-(a-1))
\]
\end{lem}
\begin{proof}
We shall show that for $\lambda\geq 1/\sqrt{a}$, the difference in the $\atanh$ terms is nonpositive. 
We then immediately have the upper bound
\[
\mu(a\lambda)-\mu(\lambda)
\leq \frac1{2}\ln\left(1+\frac{(a^2-1)\lambda^2}{1+\lambda^2}\right)
<\ln(a)\leq a-1.  
\]
On the other hand, because $\lambda\geq 1/\sqrt{a}$, the $\ln$ contribution also has the lower bound
\begin{gather*}
\frac1{2}\ln\left(1+\frac{(a^2-1)\lambda^2}{1+\lambda^2}\right)
\geq \frac1{2}\ln\left(1+\frac{(a^2-1)(1/a)}{1+1/a}\right)
=\frac1{2}\ln\left(1+(a-1)\right)\\
\geq \frac{(a-1)}{2}\left(1-\frac{a-1}{2}\right)
\end{gather*}
using a 2nd order Taylor series with Lagrange remainder in the last line, recalling $a\geq 1$ here.

For the lower bound for $\mu(a\lambda)-\mu(\lambda)$, it remains to control how negative the $\atanh$ contribution is. 
With
\[
u=\frac{\sqrt{2a\lambda}}{1+a\lambda}
\qtext{and}
v=\frac{\sqrt{2\lambda}}{1+\lambda}, 
\]
we can use the atanh addition formula~\ref{lem:ATanhAdditionFormula}, 
\[
\atanh(u)-\atanh(v)=\atanh(u)+\atanh(-v)=\atanh\left(\frac{u+(-v)}{1+u(-v)}\right)
\]
for $u,v\in(-1,1)$, which is the case for us here.
After some simplification, we recover
\[
\frac{u-v}{1-uv}=(\sqrt{a}-1)\sqrt{2\lambda}\frac{1-\lambda\sqrt{a}}{(1-\lambda\sqrt{a})^2+\lambda(1+a)}
\]
which is negative for $\lambda\geq 1/\sqrt{a}$. 
Because atanh is an odd function, taking it of the above gives a negative contribution for such $\lambda$. 
Use the AM-GM inequality to upper bound
\[
-\frac{u-v}{1-uv}\leq \frac{\sqrt{a}-1}{\sqrt{2}\sqrt{1+a}}=:w, 
\]
then use the estimate
\[
\atanh(w)\leq \frac{w}{1-w^2}
=\frac{(\sqrt{a}-1)\sqrt{2}\sqrt{1+a}}{(1+\sqrt{a})^2}
\leq \frac{\sqrt{a}-1}{2}
\]
as the remaining factor is seen to be decreasing for $a\geq 1$ upon taking logarithms. 
Using $\sqrt{a}\leq 1+(a-1)/2$, we finally have. 
\[
\mu(a\lambda)-\mu(\lambda)\geq \frac{(a-1)}{2}\left(1-\frac{a-1}{2}\right)
-\frac{\sqrt{a}-1}{2}
\geq \frac{a-1}{4}(1-(a-1)). 
\]
\end{proof}

\subsection{2nd Moment}
To estimate the 2nd moment $\E\xi^2(\lambda\abs{X})$, note that
for any $a,b>0$, the AM-GM inequality gives $(a+b)^2\leq 2(a^2+b^2)$, so that
\begin{align*}
\E\xi^2(\lambda\abs{X})
&=\E\left(\ln(1+\sqrt{\lambda\abs{X}})+\frac1{2}\ln(1+\lambda\abs{X})\right)^2\\
&\leq \E\left(2\ln^2(1+\sqrt{\lambda\abs{X}})
    +\frac1{2}\ln^2(1+\lambda\abs{X})\right). 
\end{align*}
It turns out this last expression also arises from a contour integral.

\begin{lem}\label{XiSecondMomentEstimate}
If $\lambda>0$ and $X\drawn\Cauchy{1}$, then
\begin{gather*}
\E\left(2\ln^2(1+\sqrt{\lambda\abs{X}})+\frac1{2}\ln^2(1+\lambda\abs{X})\right)\\
=2\E\arctan^2(\sqrt{\lambda\abs{X}})+\mu^2(\lambda)-\big(\arctan(\lambda)-h(\sqrt{\lambda})\big)^2
\end{gather*}
with
\[
h(\sqrt{\lambda})=\frac{\pi}{2}+\arctan\left(\frac{\sqrt{\lambda}}{\sqrt{2}}-\frac1{\sqrt{2\lambda}}\right). 
\]
\end{lem}
\begin{proof}
The computations will be a bit more involved than those for the first moment.
Starting from lemma~\ref{lem:ContourIntegralSetup} with $b=2$,
\begin{align*}
&\E\ln^2(1+\sqrt{\lambda\abs{X}})\\
&=\ln^2(1+\sqrt{i\lambda})+\ln^2(1+\sqrt{-i\lambda})\\
    &\qquad-\frac1{2}\E\ln^2(1+i\sqrt{\lambda\abs{X}})
    -\frac1{2}\E\ln^2(1-i\sqrt{\lambda\abs{X}}), 
\end{align*}
that is,
\begin{gather*}
\E 2\ln^2(1+\sqrt{\lambda\abs{X}})+\E\ln^2(1+i\sqrt{\lambda\abs{X}})
    +\E\ln^2(1-i\sqrt{\lambda\abs{X}})\\
=2\ln^2(1+\sqrt{i\lambda})+2\ln^2(1+\sqrt{-i\lambda}).
\end{gather*}
By lemma~\ref{lem:LogShiftSumOfSquares}, 
\begin{gather*}
\E\ln^2(1+i\sqrt{\lambda\abs{X}})+\E\ln^2(1-i\sqrt{\lambda\abs{X}})\\
=\E\frac1{2}\ln^2(1+(\sqrt{\lambda\abs{X}})^2)
    -2\E\arctan^2(\sqrt{\lambda\abs{X}})\\
=\E\frac1{2}\ln^2(1+\lambda\abs{X})-2\E\arctan^2(\sqrt{\lambda\abs{X}}). 
\end{gather*}
For the residue terms, we use lemma~\ref{lem:Residues2ndMomentLogSquareRoot}:
\begin{gather*}
2\ln^2(1+\sqrt{i\lambda})+2\ln^2(1+\sqrt{-i\lambda})\\
=\frac1{4}\ln^2(1+\lambda^2)-\arctan^2(\lambda)\\
    \qquad+\ln(1+\lambda^2)g(\sqrt{\lambda})+2\arctan(\lambda)h(\sqrt{\lambda})+g^2(\sqrt{\lambda})-h^2(\sqrt{\lambda})
\end{gather*}
with
\[
g(\sqrt{\lambda})=\atanh\left(\frac{\sqrt{2\lambda}}{1+\lambda}\right)
\qtext{and}
h(\sqrt{\lambda})=\frac{\pi}{2}+\arctan\left(\frac{\sqrt{\lambda}}{\sqrt{2}}-\frac1{\sqrt{2\lambda}}\right). 
\]
Recalling our computation of $\mu(\lambda)$ in lemma~\ref{lem:XiFirstMoment}, 
we can further simplify:
\begin{gather*}
2\ln^2(1+\sqrt{i\lambda})+2\ln^2(1+\sqrt{-i\lambda})\\
=\left(\frac1{4}\ln^2(1+\lambda^2)+\ln(1+\lambda^2)g(\sqrt{\lambda})+g^2(\sqrt{\lambda})\right)\\
    \qquad-\arctan^2(\lambda)+2\arctan(\lambda)h(\sqrt{\lambda})-h^2(\sqrt{\lambda})\\
=\mu^2(\lambda)-\arctan^2(\lambda)+2\arctan(\lambda)h(\sqrt{\lambda})-h^2(\sqrt{\lambda})
\end{gather*}
Putting everything together we may conclude
\begin{gather*}
\E\left(2\ln^2(1+\sqrt{\lambda\abs{X}})+\frac1{2}\ln^2(1+\lambda\abs{X})\right)\\
=2\E\arctan^2(\sqrt{\lambda\abs{X}})+\mu^2(\lambda)-\big(\arctan(\lambda)-h(\sqrt{\lambda})\big)^2.
\end{gather*}
\end{proof}

\begin{cor}[The Variance Is Bounded]\label{rem:VarianceBounded}
For $\lambda>0$ and $X\drawn\Cauchy{1}$, 
\[
\Var(\xi(\lambda\abs{X}))
\leq \min\set{2\E\ln(1+\lambda\abs{X}),\frac{\pi^2}{2}}. 
\]
\end{cor}
\begin{proof}
Just note that for $\nu>0$, 
\[
\arctan^2(\sqrt{\nu})\leq \min\set{\ln(1+\nu),\frac{\pi^2}{4}}. 
\]
The constant follows from $\arctan(x)\leq \pi/2$ for all $x\in\R$, while the $\ln(1+\nu)$ bound follows from comparing derivatives, noting that both functions take 0 when $\nu=0$. 
\end{proof}

For quantitative estimates for the 2nd moment and the variance, we make the $\E\ln(1+\lambda\abs{X})$ term explicit in the above bound. 
\begin{lem}\label{lem:ExplicitEpxpectationLogLinearSimplerFunctions}
For $\lambda\geq 0$ and $X\drawn\Cauchy{1}$, 
\[
\E\ln(1+\lambda\abs{X})
=-\frac{2}{\pi}\ln(\lambda)\arctan(\lambda)+\frac1{2}\ln(1+\lambda^2)
    +\frac{2}{\pi}\Ti_2(\lambda). 
\]
\end{lem}
\begin{proof}
From lemma~\ref{lem:ExactMomentsLogLinear}
\begin{gather*}
\E\ln(1+\lambda\abs{X})
=\frac{2}{\pi}\frac1{2i}(\Li_2(1+i\lambda)-\Li_2(1-i\lambda))
\end{gather*}
We use the reflection formula~\ref{lem:DilogarithmRelationShift1} to expand the dilogarithm terms.

Recall from lemma~\ref{lem:DilogarithmRelationShift1}, for $z\in(\C-\R)\cup(0,1)$, 
\[
\Li_2(z)+\Li_2(1-z)-\Li_2(1)=-\ln(z)\ln(1-z). 
\]
Consequently, using definition~\ref{defn:InverseTangentIntegral} for $\Ti_2$,
\begin{gather*}
\frac1{2i}(\Li_2(1+i\lambda)-\Li_2(1-i\lambda))\\
=\frac1{2i}
    \big(-\ln(-i\lambda)\ln(1+i\lambda)-\Li_2(-i\lambda)+\Li_2(1)\big)\\
    \quad-\frac1{2i}\big(-\ln(i\lambda)\ln(1-i\lambda)-\Li_2(i\lambda)
        +\Li_2(1)\big)\\
=\frac1{2i}\ln(\lambda)(\ln(1-i\lambda)-\ln(1+i\lambda))
    +\frac{\pi}{4}(\ln(1-i\lambda)+\ln(1+i\lambda))+\Ti_2(\lambda)
\end{gather*}
By lemma~\ref{lem:PolyLogInputSquared} (really the remark there) and the definition of arctan, 
\[
\frac1{2i}(\Li_2(1+i\lambda)-\Li_2(1-i\lambda))
=-\ln(\lambda)\arctan(\lambda)+\frac{\pi}{4}\ln(1+\lambda^2)+\Ti_2(\lambda). 
\]
Thus, 
\begin{gather*}
\E\ln(1+\lambda\abs{X})
=\frac{2}{\pi}\frac1{2i}(\Li_2(1+i\lambda)-\Li_2(1-i\lambda))\\
=-\frac{2}{\pi}\ln(\lambda)\arctan(\lambda)+\frac1{2}\ln(1+\lambda^2)
    +\frac{2}{\pi}\Ti_2(\lambda).
\end{gather*}
\end{proof}

\begin{cor}\label{cor:2ndMomentRatioBound}
For $0<\lambda<2$ 
\begin{gather*}
\frac{\E\xi^2(\lambda\abs{X})}{\mu(\lambda)^2}
\leq\begin{cases}
1+2\left(\lambda+\frac{4}{\pi}\big(1-\ln(\lambda)\big)\right)&\text{for $\lambda\in(0,1]$}\\
1+\frac{9}{2}\left(\lambda+\frac{4}{\pi}\right)&\text{for $\lambda\in(1,2)$}
\end{cases}. 
\end{gather*}
\end{cor}
\begin{proof}
By corollary~\ref{rem:VarianceBounded} and lemma~\ref{lem:ExplicitEpxpectationLogLinearSimplerFunctions}, 
we have
\begin{gather*}
\E\xi^2(\lambda\abs{X})\leq \ln(1+\lambda^2)+\frac{4}{\pi}\Ti_2(\lambda)-\frac{4}{\pi}\ln(\lambda)\arctan(\lambda)+\mu^2(\lambda)\\
\leq \lambda^2+\frac{4}{\pi}\lambda-\frac{4}{\pi}\lambda\ln(\lambda)
    +\mu^2(\lambda)
\end{gather*}
because $\Ti_2(\lambda)$ is an alternating series with terms of decreasing magnitude for $\lambda<2$ and that for $\lambda\leq 1$, $\ln(\lambda)$ is nonnegative.
For $\lambda\in(1,2)$, we can drop the $\ln(\lambda)$ term for an upper bound.
Consequently, using $\mu(\lambda)\geq \sqrt{2\lambda}/(1+\lambda)$ from remark~\ref{rem:MuApproxSmallScales}, 
\[
\frac{\E\xi^2(\lambda\abs{X})}{\mu(\lambda)^2}
\leq 1+2\left(\lambda+\frac{4}{\pi}\big(1-\ln(\lambda)\big)\right)
\]
for $\lambda\leq 1$, 
and
\[
\frac{\E\xi^2(\lambda\abs{X})}{\mu(\lambda)^2}
\leq 1+\frac{9}{2}\left(\lambda+\frac{4}{\pi}\right)
\]
for $\lambda\in(1,2)$. 
\end{proof}

\begin{lem}\label{lem:LogShiftSumOfSquares}
For $r>0$, 
\[
\ln^2(1+i r)+\ln^2(1-i r)=\frac1{2}\ln^2(1+r^2)-2\arctan^2(r). 
\]
\end{lem}
\begin{proof}
We are adding complex conjugates, so the left hand side is
\begin{align*}
2\Re\ln^2(1+ir)
&=2\Re\left(\frac1{2}\ln(1+r^2)+i\arctan(r)\right)^2\\
&=2\left(\frac1{4}\ln^2(1+r^2)-\arctan^2(r)\right)
=\frac1{2}\ln^2(1+r^2)-2\arctan^2(r). 
\end{align*}
\end{proof}

\begin{lem}\label{lem:Residues2ndMomentLogSquareRoot}
For $\nu>0$, 
\begin{align*}
&\ln^2(1+\nu\sqrt{i})+\ln^2(1+\nu\sqrt{-i})\\
&=\frac1{8}\ln^2(1+\nu^4)-\frac1{2}\arctan^2(\nu^2)\\
    &\qquad+\frac1{2}\ln(1+\nu^4)g(\nu)+\arctan(\nu^2)h(\nu)
    +\frac1{2}(g^2(\nu)-h^2(\nu))
\end{align*}
with
\[
g(\nu)=\atanh\left(\frac{\nu\sqrt{2}}{1+\nu^2}\right)
\qtext{and}
h(\nu)=\frac{\pi}{2}+\arctan\left(\frac{\nu}{\sqrt{2}}-\frac1{\nu\sqrt{2}}\right). 
\]
\end{lem}
\begin{proof}
Using lemma~\ref{lem:LogToAtanhLogSplit},
\begin{align*}
&\ln^2(1+\nu\sqrt{i})
=\left(\atanh(\nu\sqrt{i})+\frac1{2}\ln(1-i\nu^2)\right)^2\\
&=\atanh^2(\nu\sqrt{i})+\atanh(\nu\sqrt{i})\ln(1-i\nu^2)+\frac1{4}\ln^2(1-i\nu^2)
\end{align*}
and similarly
\begin{align*}
&\ln^2(1+\nu\sqrt{-i})
=\left(\atanh(\nu\sqrt{-i})+\frac1{2}\ln(1+i\nu^2)\right)^2\\
&=\atanh^2(\nu\sqrt{-i})+\atanh(\nu\sqrt{-i})\ln(1+i\nu^2)+\frac1{4}\ln^2(1+i\nu^2)
\end{align*}
Adding yields several terms:
\begin{align*}
\boxed{1}(\nu)&:=\frac1{4}\ln^2(1+i\nu^2)+\frac1{4}\ln^2(1-i\nu^2)\\
\boxed{2}(\nu)&:=\atanh(\nu\sqrt{i})\ln(1-i\nu^2)+\atanh(\nu\sqrt{-i})\ln(1+i\nu^2)\\
\boxed{3}(\nu)&:=\atanh^2(\nu\sqrt{i})+\atanh^2(\nu\sqrt{-i})
\end{align*}

From lemma~\ref{lem:LogShiftSumOfSquares}, 
\[
\boxed{1}(\nu)=\frac1{4}\left(\frac1{2}\ln^2(1+\nu^4)-2\arctan^2(\nu^2)\right)
=\frac1{8}\ln^2(1+\nu^4)-\frac1{2}\arctan^2(\nu^2). 
\]
We also have
\begin{align*}
\boxed{2}(\nu)
&=\atanh(\nu\sqrt{i})\left(\frac1{2}\ln(1+\nu^4)-i\arctan(\nu^2)\right)\\
    &\qquad
    +\atanh(\nu\sqrt{-i})\left(\frac1{2}\ln(1+\nu^4)+i\arctan(\nu^2)\right)\\
&=\frac1{2}\ln(1+\nu^4)\big(\atanh(\nu\sqrt{i})+\atanh(\nu\sqrt{-i})\big)\\
    &\qquad-i\arctan(\nu^2)\big(\atanh(\nu\sqrt{i})-\atanh(\nu\sqrt{-i})\big)\\
&=\frac1{2}\ln(1+\nu^4)g(\nu)+\arctan(\nu^2)h(\nu). 
\end{align*}

Let 
\begin{align*}
g(\nu)&:=\atanh(\nu\sqrt{i})+\atanh(\nu\sqrt{-i})\\
&=\atanh\left(\frac{\nu(\sqrt{i}+\sqrt{-i})}{1+\nu^2\sqrt{-i^2}}\right)
=\atanh\left(\frac{\nu\sqrt{2}}{1+\nu^2}\right)
\end{align*}
by the atanh addition formula~\ref{lem:ATanhAdditionFormula}, as $\sqrt{\pm i}=(1\pm i)/\sqrt{2}$ are conjugates of each other.

Let 
\[
h(\nu):=-i\big(\atanh(\nu\sqrt{i})-\atanh(\nu\sqrt{-i})\big). 
\]
Then 
\begin{align*}
&g^2(\nu)-h^2(\nu)\\
&=\atanh^2(\nu\sqrt{i})+\atanh^2(\nu\sqrt{-i})+2\atanh(\nu\sqrt{i})\atanh(\nu\sqrt{-i})\\
    &\qquad+\big(\atanh(\nu\sqrt{i})-\atanh(\nu\sqrt{-i})\big)^2\\
&=2\big(\atanh^2(\nu\sqrt{i})+\atanh^2(\nu\sqrt{-i})\big)
=2\boxed{3}(\nu). 
\end{align*}
So we are left to understand $h(\nu)$. 
By lemma~\ref{lem:HNuAcrossSingularity}, it is
\[
h(\nu)
=\frac{\pi}{2}+\arctan\left(\frac{\nu}{\sqrt{2}}-\frac1{\nu\sqrt{2}}\right). 
\]
\end{proof}

\begin{lem}\label{lem:HNuAcrossSingularity}
For $\nu>0$,
\[
h(\nu):=-i\big(\atanh(\nu\sqrt{i})-\atanh(\nu\sqrt{-i})\big)
=\frac{\pi}{2}+\arctan\left(\frac{\nu}{\sqrt{2}}-\frac1{\nu\sqrt{2}}\right). 
\]
\end{lem}
\begin{rem}
For $\nu<1$, we can rewrite the above as
\[
\frac{\pi}{2}-\arctan\left(\frac{1-\nu^2}{\nu\sqrt{2}}\right)
=\arctan\left(\frac{\nu\sqrt{2}}{1-\nu^2}\right). 
\]
\end{rem}
\begin{proof}
We cannot directly use the atanh addition formula because there is a singularity when $\nu$ crosses 1. 
However, by definition of atanh~\ref{defn:ATanh}, we can convert $h(\nu)$ as follows, using $\sqrt{-i}=-i\sqrt{i}$ 
\begin{align*}
h(\nu)&:=-i\big(\atanh(\nu\sqrt{i})-\atanh(\nu\sqrt{-i})\big)\\
&=-i\big(-i\arctan(i\nu\sqrt{i})-(-i)\arctan(i\nu\sqrt{-i})\big)\\
&=-\arctan(i\nu\sqrt{i})+\arctan(\nu\sqrt{i}). 
\end{align*}
We now use the inversion formula~\ref{lem:ArcTanInversionFormula} for $\arctan$. 
\begin{align*}
h(\nu)&=-\arctan(i\nu\sqrt{i})+\frac{\pi}{2}-\arctan(1/(\nu\sqrt{i}))\\
&=\frac{\pi}{2}-\big(\arctan(i\nu\sqrt{i})+\arctan(-i\sqrt{i}/\nu)\big)
\end{align*}
The following identity holds
\[
\big(\arctan(i\nu\sqrt{i})+\arctan(-i\sqrt{i}/\nu)\big)
=-\arctan\left(\frac{\nu}{\sqrt{2}}-\frac1{\nu\sqrt{2}}\right), 
\]
because both analytic expressions are 0 at $\nu=1$, and their derivatives match for $\nu>0$. 
\end{proof}

\section{Polylogarithms and Their Friends}\label{app}
The polylogarithms $\Li_b(z)$ arise when we compute or estimate the first and second moments of the coordinate projections; they will help us give quantitative bounds which are needed in some of the proofs. 
References for polylogarithms are~\citep{LewinPolylogarithms1981} and~\citep{MaximonDilogarithm2003}. 

As initial motivation for studying such functions, we have the following lemma.
\begin{lem}\label{lem:ExactMomentsLogLinear}
Let $X\drawn\Cauchy{1}$ and $\nu>0$. 
Then for $b>-1$, 
\[
\E\ln^b(1+\nu\abs{X})
=\frac{\Gamma(b+1)}{i\pi}(\Li_{b+1}(1+i\nu)-\Li_{b+1}(1-i\nu)). 
\]
\end{lem}
\begin{proof}
We have
\[
I_b(\nu):=\E\ln^b(1+\nu\abs{X})=\frac{2}{\pi}\int_0^\infty\frac{\ln^b(1+\nu x)}{1+x^2}\,dx
\]
Change variables 
$u=1+\nu x$ and then $t=\ln(u)$ to find 
\begin{align*}
I_b(\nu)=\frac{2\nu}{\pi}\int_0^\infty \frac{t^be^t}{(e^t-(1+i\nu))(e^t-(1-i\nu))}\,dt. 
\end{align*}
Using partial fractions, we may write
\begin{align*}
I_b(\nu)&=\frac1{i\pi}\int_0^\infty t^b \frac{2i\nu e^t}
    {(e^t-(1+i\nu))(e^t-(1-i\nu))}\,dt\\
&=\frac1{i\pi}\int_0^\infty \frac{t^b(1+i\nu)}{e^t-(1+i\nu)}
    -\frac{t^b(1-i\nu)}{e^t-(1-i\nu)}\,dt\\
&=\frac{\Gamma(b+1)}{i\pi}(\Li_{b+1}(1+i\nu)-\Li_{b+1}(1-i\nu)). 
\end{align*}
by definition~\ref{defn:PolyLog}. 
The polylogarithms are defined because $\nu>0$, and if $b>0$, the value at $\nu=0$ is also defined. 
\end{proof}

General references for complex analysis are~\citep{SteinShakarchiComplex2003} for proofs and~\citep{NeedhamVisual1997} for intuition. 
If $z=x+iy\in\C$ with $x,y\in\R$, then $\Re(z):=x$ and $\Im(z):=y$. 
If $z=re^{i\theta}=x+iy\in\C$, denote $z^\ast=re^{-i\theta}=x-iy$ for the complex conjugate. 
Further $\abs{z}^2=zz^\ast=x^2+y^2$. 
Thus, if $w=se^{i\phi}$, we have
\[
(zw)^\ast=(rse^{i(\theta+\phi)})^\ast=rse^{-i(\theta+\phi)}=z^\ast w^\ast. 
\]
Further, if $w\neq 0$, 
\[
\abs{\frac{z}{w}}^2=\frac{zz^\ast}{ww^\ast}=\frac{r^2}{s^2}=\frac{\abs{z}^2}{\abs{w}^2}. 
\]

For us, analytic functions are synonymous with holomorphic ones.
We shall be using two theorems from complex analysis repeatedly. Cf.~\citep[page~52,96]{SteinShakarchiComplex2003}. 
\begin{thm}[Analytic Continuation]
Let $f$ and $g$ be analytic functions in a connected open subset $\Omega$ of $\C$. 
If $f(z)=g(z)$ for all $z$ in a non-empty open subset of $\Omega$, then $f(z)=g(z)$ throughout $\Omega$. 
\end{thm}

\begin{thm}[Primitives]
Let $f$ be an analytic function in a simply connected subset $\Omega$ of $\C$. 
Then for $z_0,z\in\Omega$, the function
\[
F(z):=\int_{z_0}^z f(w)\,dw=\int_\gamma f(w)\,dw
\]
is analytic too, with $\gamma$ any path from $z_0$ to $z$ lying in $\Omega$. 
\end{thm}

\begin{defn}[The Logarithm]
For $z=re^{i\theta}\in \C-(-\infty,0]$, define (the principle branch of) the logarithm of $z$, $\ln(z)$ as
\[
\ln(z):=\ln(r)+i\theta=\int_1^z\,\frac{dw}{w}  
\]
for any path from 1 to $z$ in $\C-(-\infty,0]$. 
\end{defn}
\begin{rem}
Note that
$\ln(z^\ast)=\ln(r)-i\theta=\ln(z)^\ast$. 
The map $w\mapsto 1/w$ takes $\C-(-\infty,0]$ to itself; for if $w=se^{i\phi}$, with $\abs{\phi}<\pi$, then $1/w=(1/s)e^{-i\phi}$ which also lives in $\C-(-\infty,0]$. 
With this choice of principle branch, the logarithm still satisfies
$-\ln(1/w)=\ln(w)$
via
\[
-\ln(1/w)=-(\ln(1/s)+i(-\phi))=\ln(s)+i\phi=\ln(w). 
\]

Similarly, note that if $\Re(z),\Re(w)>0$, then 
$zw=rse^{i(\theta+\phi)}$ with $\abs{\theta+\phi}<\pi$
so $\arg(zw)=\theta+\phi$ and
\[
\ln(zw)=\ln(rs)+i(\theta+\phi)=\ln(z)+\ln(w)
\]
in this case.
However, the general identity $\ln(z_1z_2)=\ln(z_1)+\ln(z_2)$ does not hold.
\end{rem}

\begin{defn}[The Polylogarithm of Order 1]
Define the polylogarithm of order 1, $\Li_1(z)$ as
\[
\Li_1(z):=\sum_{j=1}^\infty \frac{z^j}{j} \qtext{for} \abs{z}<1
\]
and
\[
\Li_1(z):=-\ln(1-z)=\ln\left(\frac1{1-z}\right) \qtext{for} z\in\C-[1,\infty). 
\]
\end{defn}
For general $z$, the domain makes sense, as $1-z=-(z-1)\in\C-(-\infty,0]$ for the $z$ in question. 
Recall when $\abs{z}<1$, 
\[
-\ln(1-z)=\sum_{j=1}^\infty \frac{z^j}{j}, 
\]
noting that both sides agree when $z=0$, and upon differentiating, 
\[
\frac{d}{dz}\sum_{j=1}^\infty \frac{z^j}{j}
=\sum_{j=0}^\infty z^j=\frac1{1-z}=\frac{d}{dz}(-\ln(1-z)) 
\]
which means $-\ln(1-z)$ and the sum differ by a constant, namely 0. 

The order of the polylogarithms may be extended; the general integral form below will be useful for some of the computations later.  
\begin{defn}\label{defn:PolyLog}
For $b>0$, define the \emph{polylogarithm of order $b$} as
\[
\Li_b(z):=\sum_{j=1}^\infty \frac{z^j}{j^b} \qtext{for} \abs{z}<1
\]
and
\begin{gather*}
\Li_b(z):=\frac1{\Gamma(b)}\int_0^\infty \frac{zt^{b-1}}{e^t-z}\,dt
=\frac1{\Gamma(b)}\int_0^\infty \frac{zt^{b-1}e^{-t}}{1-e^{-t}z}\,dt. 
\end{gather*}
for $z\in\C-[1,\infty)$.  
\end{defn}
The nonintegral order polylogarithms also extend to the unit circle when the order is greater than 1. 
\begin{lem}\label{lem:PolyLogOnUnitCircle}
For $b>1$ and $z\in\C$ with $\abs{z}=1$, 
\[
\Li_b(z)<b. 
\]
\end{lem}
\begin{proof}
By definition, 
\[
\Li_b(z)=\sum_{j=1}^\infty \frac{z^j}{j^b}
\qtext{so that when $\abs{z}=1$, }
\abs{\Li_b(z)}\leq \sum_{j=1}^\infty \frac{\abs{z}^j}{j^b} 
=\sum_{j=1}^\infty \frac1{j^b}
\]
The series is finite because $b>1$; 
concretely, by the integral test (because $1/x^b$ is convex), 
\begin{gather*}
\sum_{j=1}^\infty \frac1{j^b}=1+\sum_{j=2}^\infty \frac1{j^b}
\leq 1+\int_1^\infty \frac1{x^b}\,dx
=1+(b-1)\frac{-1}{x^{b-1}}\rvert_1^\infty
=b<\infty. 
\end{gather*}
\end{proof}
\begin{lem}\label{lem:PolyLogInputSquared}
For $z\in(\C-\R)\cup (-1,1)$ and $b>0$, 
\[
\Li_b(z)+\Li_b(-z)=\frac1{2^{b-1}}\Li_b(z^2). 
\]
If $b>1$, the equality also holds when $z=\pm 1$. 
\end{lem}
\begin{rem}\label{rem:LogShiftPM}
When $b=1$, recover
\[
\ln(1-z)+\ln(1+z)=-\big(\Li_1(z)+\Li_1(-z)\big)=-\Li_1(z^2)=\ln(1-z^2). 
\]
\end{rem}
\begin{proof}
First assume $\abs{z}<1$. 
From the power series,
\begin{gather*}
\Li_b(z)+\Li_b(-z)=\sum_{j=1}^\infty \frac{z^j+(-z)^j}{j^b}
=\sum_{j=1}^\infty z^j\frac{1+(-1)^j}{j^b}\\
=2\sum_{j=1}^\infty \frac{z^{2j}}{(2j)^b}
=\frac1{2^{b-1}}\sum_{j=1}^\infty \frac{(z^2)^j}{j^b}
=\frac1{2^{b-1}}\Li_b(z^2). 
\end{gather*}
Both sides are analytic functions on $(\C-\R)\cup (-1,1)$, so by analytic continuation, the identity continues to hold there. 
If $b>1$, the power series are also defined at $z=\pm 1$. 
\end{proof}

%
A useful property of the polylogarithms and the logarithm that we shall use repeatedly in computations is that they are all symmetric about the real axis, that is,
$\Li_b(z^\ast)^\ast=\Li_b(z)$ 
or concretely
\[
\Re\Li_b(z^\ast)=\Re\Li_b(z) \qtext{and}
\Im\Li_b(z^\ast)=-\Im\Li_b(z). 
\]
Powers and polynomials of such functions also have this property.
Intuitively this symmetry follows from the real coeffecients in their power series expansions, so that $\Li(x)\in\R$ when $x<1$. 
Rigorously, we use the Schwarz reflection principle; because $\Li_b(z)$ is analytic in $\C-[1,\infty)$ when $0\leq \arg(z)<\pi$ and real valued on $(-\infty, 1)$, $\Li_b(z)$ may be extended to the rest of $\C-[1,\infty)$ in an analytic fashion. 
Analytic continuation then dictates that this extension coincides with the original definition of $\Li_b(z)$. 
See~\citep{SteinShakarchiComplex2003} pages 57-59 for the Schwarz reflection principle, page 56 for showing the integral definitions of $\Li_b(z)$ are analytic, and page 52 for the principle of analytic continuation. 

\subsection{Arctan and the Inverse Tangent Integrals} 
The function $t\mapsto\arctan(t)$ is proportional to the distribution function of $\abs{X}$ with $X\drawn\Cauchy{1}$. 
It is then perhaps not surprising that $\arctan$ and its relatives arise in working with functions of Cauchy random variables. 
We outline the properties we shall be using here.

The following definition is opaque but most useful to us. 
\begin{defn}\label{defn:ComplexArcTan}
Define $\arctan(z)$ as
\[
\arctan(z):=\sum_{j=0}^\infty (-1)^j\frac{z^{2j+1}}{2j+1}
\qtext{for} \abs{z}<1,
\]
and 
\[
\arctan(z)=\int_0^z\frac1{1+w^2}\,dw
\qtext{for} z\in(\C-i\R)\cup (-i,i). 
\]
Equivalently,
\[
\arctan(z):=\frac1{2i}(\ln(1+iz)-\ln(1-iz))=\frac1{2i}(\Li_1(iz)-\Li_1(-iz)). 
\]
\end{defn}
\begin{rem}
From the integral formulation, we also immediately have, with $v=-w$, 
\[
\arctan(-z)=\int_0^{-z}\frac1{1+w^2}\,dw=-\int_0^z\frac1{1+(-v)^2}\,dv
=-\arctan(z). 
\]
The last definition for $\arctan(z)$ follows from 
\begin{gather*}
\frac{d}{dz}\frac1{2i}(\ln(1+iz)-\ln(1-iz))
=\frac1{2i}\left(\frac{i}{1+iz}-\frac{(-i)}{1-iz}\right)\\
=\frac1{2}\left(\frac1{1+iz}+\frac1{1-iz}\right)
=\frac1{1+z^2}=\frac{d}{dz}\arctan(z)  
\end{gather*}
and that $\arctan(0)=0$.
\end{rem}

We can generalize. 
\begin{defn}\label{defn:InverseTangentIntegral}
For $z\in\C-i\R\cup(-i,i)$ and $b>0$, define
the \emph{inverse tangent integral of order} $b$ as
\[
\Ti_b(z):=\sum_{j=0}^\infty (-1)^j\frac{z^{2j+1}}{(2j+1)^b}\qtext{for} \abs{z}<1
\]
and
\[
\Ti_b(z)=\frac{\Li_b(iz)-\Li_b(-iz)}{2i} \qtext{for} z\in\C-i\R\cup(-i,i). 
\]
\end{defn}
\begin{rem}
Note if $\abs{y}<1$, we find
\begin{gather*}
\Li_b(iy)-\Li_b(-iy)=\sum_{j=1}^\infty \frac{(iy)^j-(-iy)^j}{j^b}
=\sum_{j=1}^\infty i^j\frac{y^j}{j^b}(1-(-1)^j)\\
=2\sum_{j=0}^\infty i^{2j+1}\frac{y^{2j+1}}{(2j+1)^b}
=2i\sum_{j=0}^\infty (-1)^j\frac{y^{2j+1}}{(2j+1)^b}
=:2i\Ti_b(y)\in i\R. 
\end{gather*}
Hence,
\[
\Ti_b(y)=\frac{\Li_b(iy)-\Li_b(-iy)}{2i} 
\]
when $\abs{y}<1$ and $b>0$. 
The right hand side continues to make sense for $y\in(\C-i\R)\cup(-i,i)$, so we may define 
\[
\Ti_b(z):=\frac{\Li_b(iz)-\Li_b(-iz)}{2i}
\]
as an analytic function on $z\in(\C-i\R)\cup (-i,i)$ that agrees with the power series on the interior of the unit circle. 
\end{rem}
\begin{rem}
In particular, we have $\Ti_1(z)=\arctan(z)$. 
\end{rem}

To focus on the behavior of $\arctan$ on $(-i,i)$ which was not addressed in the inversion formula~\ref{lem:ArcTanInversionFormula}, we change points of view through a rotation of the complex plane. 
\begin{defn}\label{defn:ATanh}
Define the function $\atanh$ as
\[
\atanh(x)=\sum_{j=0}^\infty \frac{x^{2j+1}}{2j+1}
\qtext{for} \abs{x}<1,
\]
and as
\[
\atanh(z)=\int_0^z\frac1{1-w^2}\,dw=-i\arctan(iz)
\qtext{for} z\in(\C-\R)\cup (-1,1). 
\]
or equivalently as
\[
\atanh(z)=\frac1{2}(\ln(1+z)-\ln(1-z))=\frac1{2}(\Li_1(z)-\Li_1(-z)). 
\]
\end{defn}
To see that the definitions are consistent, 
note first from the power series, $\atanh(0)=0=\arctan(0)$, 
while on the other hand,
\[
\frac{d}{dz}(-i)\arctan(iz)=\frac{(-i)}{1+(iz)^2}(i)=\frac1{1-z^2}
=\frac{d}{dz}\atanh(z). 
\]

\begin{lem}\label{lem:LogToAtanhLogSplit}
Let $z\in(\C-\R)\cup (-1,1)$ 
then
\[
\ln(1+z)=\atanh(z)+\frac1{2}\ln(1-z^2). 
\]
\end{lem}
\begin{proof}
Just split into even and odd degree terms.
\begin{gather*}
\Li_1(z)=\sum_{j=1}^\infty \frac{z^j}{j}
=\sum_{j=0}^\infty \frac{z^{2j+1}}{(2j+1)}
    +\sum_{j=1}^\infty \frac{z^{2j}}{(2j)}
=\atanh(z)+\frac1{2}\sum_{j=1}^\infty \frac{(z^2)^j}{j}\\
=\atanh(z)+\frac1{2}\Li_1(z^2). 
\end{gather*}
The equality extends to $(\C-\R)\cup(-1,1)$ as both sides are analytic there.
We now have
\[
\ln(1+z)=-\Li_b(-z)=-\atanh(-z)+\frac1{2}\ln(1-(-z)^2)
=\atanh(z)+\frac1{2}\ln(1+z^2) 
\]
as desired. 
\end{proof}

Here is the addition formula. 
\begin{lem}[Atanh Addition Formula]\label{lem:ATanhAdditionFormula}
If $-1<x,y<1$, 
\[
\atanh(x)+\atanh(y)=\atanh\left(\frac{x+y}{1+xy}\right). 
\]
If $z\in\C-\R$, 
\[
\atanh(z)+\atanh(z^\ast)=\atanh\left(\frac{2\Re(z)}{1+\abs{z}^2}\right)
\]
\end{lem}
\begin{proof}
Because $\atanh$ is odd, the addition formula also covers subtraction too. 
Check that 
\[
\frac{d}{dz}\atanh\left(\frac{z+w}{1+zw}\right)=\frac1{1-z^2}=\frac{d}{dz}\atanh(z). 
\]
So
\begin{gather*}
\atanh\left(\frac{z+w}{1+zw}\right)=\atanh(z)+c
\end{gather*}
with $c$ a constant. Taking $z=0$ forces $c=\atanh(w)$ as desired. 

For $z,w\in(\C-\R)\cup (-1,1)$, let
\[
f(z,w):=\frac{z+w}{1+zw}. 
\]
We want to know when $f(z,w)$ also lies in the domain of atanh. 
When $w=z^\ast$, 
\begin{gather*}
\abs{\frac{z+z^\ast}{1+zz^\ast}}
\leq \frac{2\abs{\Re(z)/\abs{z}}}{\frac1{\abs{z}}+\abs{z}}
\leq \frac{2}{\frac1{\abs{z}}+\abs{z}}\leq \frac1{\sqrt{\abs{z}/\abs{z}}}=1. 
\end{gather*}
by the AM-GM inequality. 
The equality case occurs just if $\abs{z}=1$, but in that case, 
$\abs{\Re(z)}/\abs{z}<1$
as $z=\pm 1$ is not allowed for $\atanh$. 
We are thus ok for all $z\in (\C-\R)\cup (-1,1)$ in this $w=z^\ast$ case.

When $x,y\in (-1,1)$, we may consider
\[
\partial_x f(x,y)=\frac1{1+xy}-\frac{(x+y)}{(1+xy)^2}y
=\frac1{(1+xy)^2}(1+xy-xy-y^2)=\frac{1-y^2}{(1+xy)^2}>0
\]
and by symmetry, $\partial_y f(x,y)>0$. 
So $f$ is increasing in each of the individual coordinates.
In particular, when $-1<x<y<1$, 
\[
\frac{2x}{1+x^2}=f(x,x)<f(x,y)<f(y,y)=\frac{2y}{1+y^2}.  
\]
For each permutation of $x$, $y$, and $0$, check that 
\[
\abs{f(x,y)}<\frac{2t}{1+t^2}<1
\qtext{with} t=\max\set{\abs{x},\abs{y}}, 
\]
by the AM-GM inequality, with strict inequality because $\abs{x},\abs{y}<1$.

%
%

\end{proof}

\subsection{Dilogarithm Properties}
The dilogarithm is the polylogarithm of order 2. 
\begin{lem}[Reflection Formula]\label{lem:DilogarithmRelationShift1}
For $z\in(\C-\R)\cup(0,1)$, 
\[
\Li_2(z)+\Li_2(1-z)-\Li_2(1)=-\ln(z)\ln(1-z). 
\]
\end{lem}
\begin{proof}
%
%
(Compare to~\citep[page~5]{LewinPolylogarithms1981}.)
Consider
\begin{gather*}
\frac{d}{dz}(\Li_2(z)+\Li_2(1-z))
=\frac{\Li_1(z)}{z}+\frac{\Li_1(1-z)}{1-z}(-1)
=\frac{-\ln(1-z)}{z}+\frac{\ln(z)}{1-z}. 
\end{gather*}
On the other hand,
\begin{gather*}
\frac{d}{dz}(-\ln(z)\ln(1-z))=\frac{-\ln(1-z)}{z}+\frac{\ln(z)}{1-z}. 
\end{gather*}

Because the domain $(\C-\R)\cup(0,1)$ is simply connected and the derivative above is analytic there, we have
\begin{gather*}
-\ln(z)\ln(1-z)+\ln(z_0)\ln(1-z_0)\\
=\Li_2(z)+\Li_2(1-z)-(\Li_2(z_0)+\Li_2(1-z_0)) 
\end{gather*}
for some $z_0$ which we may take to lie on $(0,1)$. 
Taking the limit as $z_0\to 0$ is safe, as the Taylor series for $\ln(1-z_0)$ ensures $\ln(z_0)\ln(1-z_0)\to 0$, while the dilogarithm is continuous on $(-\infty,1]$. 
Hence, 
\[
-\ln(z)\ln(1-z)=\Li_2(z)+\Li_2(1-z)-\Li_2(1)
\]
as desired. 
Note that proving the identity via integration by parts has to make this same limiting argument.  
\end{proof}

\subsection{Inversion Formulas}\label{sec:InversionFormulas}
The following lemma allows us to describe the survival function of $\abs{X}$ with $X\drawn\Cauchy{1}$ in a convenient way. 
Note that the survival function for $\abs{X}$  will only consider $z=x>0$. 
\begin{lem}\label{lem:ArcTanInversionFormula}
For $z\in\C-i\R$, 
\[
\arctan(z)+\arctan\left(\frac1{z}\right)
=\begin{cases}
\pi/2 &\txt{if} \Re(z)>0\\
-\pi/2 &\txt{if} \Re(z)<0. 
\end{cases}
\]
\end{lem}
\begin{rem}
On the imaginary axis, $\arctan(ir)=i\atanh(r)$ and $\atanh$ is only defined for $r\in (-1,1)$ so $1/r$ does not make sense there.
Consequently the domain in question has two connected components, so different constants should not be unexpected. 
\end{rem}
\begin{proof}
First note that the left hand side is a constant
\[
\frac{d}{dz}\left(\arctan(z)+\arctan\left(\frac1{z}\right)\right)
=\frac1{1+z^2}+\frac1{1+z^{-2}}\frac{-1}{z^2}=0. 
\]

The constant is determined by representative points $z=\pm 1$ in the right and left hand planes respectively. 
\end{proof}
\bibliographystyle{imsart-number} 

\bibliography{W2_level_l1}

\end{document}